\documentclass{amsart}

\usepackage{hyperref}  
\usepackage{fullpage}
\usepackage{comment}
\usepackage{enumitem}      
\usepackage{graphicx}   
\usepackage{ifthen}  
\usepackage[svgnames]{xcolor}  
\usepackage{amssymb}
\usepackage{newtxmath}
\usepackage{tikz}
\usepackage{tikz-cd}
\usepackage{makebox}

\DeclareMathAlphabet{\mathbbe}{U}{bbold}{m}{n}

\setlist{}
\setenumerate{leftmargin=*,labelindent=0\parindent}
\setitemize{leftmargin=\parindent}
\newtheorem{thm}{Theorem}[section]
\newtheorem{lem}[thm]{Lemma}
\newtheorem{prop}[thm]{Proposition}

\theoremstyle{definition}
\newtheorem{defn}[thm]{Definition}

\newtheorem{ex}[thm]{Example}

\theoremstyle{remark}
\newtheorem{rmk}[thm]{Remark}

\newtheorem{war}[thm]{Warning}

\usepackage{bbm}
\usepackage[bbgreekl]{mathbbol}

\makeatletter
\let\c@equation\c@thm
\makeatother
\numberwithin{equation}{section}

\newcommand{\id}{\textup{id}}
\newcommand{\cod}{\textup{cod}}
\newcommand{\dom}{\textup{dom}}

\newcommand{\To}{\Rightarrow}
\newcommand{\ev}{\textup{ev}}
\renewcommand{\AA}{\mathbb{A}}
\newcommand{\BB}{\mathbb{B}}
\newcommand{\CC}{\mathbb{C}}

\newcommand{\II}{\mathbb{I}}

\newcommand{\RR}{\mathbb{R}}

\newcommand{\UU}{\mathbb{U}}

\newcommand{\XX}{\mathbb{X}}
\newcommand{\YY}{\mathbb{Y}}
\newcommand{\ZZ}{\mathbb{Z}}
\newcommand{\1}{\mathbbe{1}}
\newcommand{\2}{\mathbbe{2}}  
\newcommand{\3}{\mathbbe{3}}
\newcommand{\4}{\mathbbe{4}}
\newcommand{\cat}[1]{\textup{\textsf{#1}}}

\newcommand{\cA}{\mathcal{A}}
\newcommand{\cB}{\mathcal{B}}
\newcommand{\cC}{\mathcal{C}}
\newcommand{\cD}{\mathcal{D}}
\newcommand{\cE}{\mathcal{E}}
\newcommand{\cF}{\mathcal{F}}

\newcommand{\cJ}{\mathcal{J}}
\newcommand{\Fib}{\cat{Fib}}
\newcommand{\TFib}{\cat{TFib}}
\newcommand{\cart}{{\cat{cart}}}
\newcommand{\Ladj}{\mathbb{L}\cat{adj}}
\newcommand{\Radj}{\mathbb{R}\cat{adj}}
\newcommand{\slice}[1]{{/{#1}}}
\newcommand{\maps}{\colon}
\newcommand{\TypeFont}[1]{\mathsf{#1}}

\makeatletter
\def\prd#1{\@ifnextchar\bgroup{\prd@parens{#1}}{%
    \@ifnextchar\sm{\prd@parens{#1}\@eatsm}{%
    \@ifnextchar\prd{\prd@parens{#1}\@eatprd}{%
    \@ifnextchar\;{\prd@parens{#1}\@eatsemicolonspace}{%
    \@ifnextchar\\{\prd@parens{#1}\@eatlinebreak}{%
    \@ifnextchar\narrowbreak{\prd@parens{#1}\@eatnarrowbreak}{%
      \prd@noparens{#1}}}}}}}}
\def\prd@parens#1{\@ifnextchar\bgroup%
  {\mathchoice{\@dprd{#1}}{\@tprd{#1}}{\@tprd{#1}}{\@tprd{#1}}\prd@parens}%
  {\@ifnextchar\sm%
    {\mathchoice{\@dprd{#1}}{\@tprd{#1}}{\@tprd{#1}}{\@tprd{#1}}\@eatsm}%
    {\mathchoice{\@dprd{#1}}{\@tprd{#1}}{\@tprd{#1}}{\@tprd{#1}}}}}
\def\@eatsm\sm{\sm@parens}
\def\prd@noparens#1{\mathchoice{\@dprd@noparens{#1}}{\@tprd{#1}}{\@tprd{#1}}{\@tprd{#1}}}
\def\lprd#1{\@ifnextchar\bgroup{\@lprd{#1}\lprd}{\@@lprd{#1}}}
\def\@lprd#1{\mathchoice{{\textstyle\prod}}{\prod}{\prod}{\prod}({\textstyle #1})\;}
\def\@@lprd#1{\mathchoice{{\textstyle\prod}}{\prod}{\prod}{\prod}({\textstyle #1}),\ }
\def\tprd#1{\@tprd{#1}\@ifnextchar\bgroup{\tprd}{}}
\def\@tprd#1{\mathchoice{{\textstyle\prod_{(#1)}}}{\prod_{(#1)}}{\prod_{(#1)}}{\prod_{(#1)}}}
\def\dprd#1{\@dprd{#1}\@ifnextchar\bgroup{\dprd}{}}
\def\@dprd#1{\prod_{#1}\,}
\def\@dprd@noparens#1{\prod_{#1}\,}

\def\@eatnarrowbreak\narrowbreak{%
  \@ifnextchar\prd{\narrowbreak\@eatprd}{%
    \@ifnextchar\sm{\narrowbreak\@eatsm}{%
      \narrowbreak}}}
\def\@eatlinebreak\\{%
  \@ifnextchar\prd{\\\@eatprd}{%
    \@ifnextchar\sm{\\\@eatsm}{%
      \\}}}
\def\@eatsemicolonspace\;{%
  \@ifnextchar\prd{\;\@eatprd}{%
    \@ifnextchar\sm{\;\@eatsm}{%
      \;}}}

\def\sm#1{\@ifnextchar\bgroup{\sm@parens{#1}}{%
    \@ifnextchar\prd{\sm@parens{#1}\@eatprd}{%
    \@ifnextchar\sm{\sm@parens{#1}\@eatsm}{%
    \@ifnextchar\;{\sm@parens{#1}\@eatsemicolonspace}{%
    \@ifnextchar\\{\sm@parens{#1}\@eatlinebreak}{%
    \@ifnextchar\narrowbreak{\sm@parens{#1}\@eatnarrowbreak}{%
        \sm@noparens{#1}}}}}}}}
\def\sm@parens#1{\@ifnextchar\bgroup%
  {\mathchoice{\@dsm{#1}}{\@tsm{#1}}{\@tsm{#1}}{\@tsm{#1}}\sm@parens}%
  {\@ifnextchar\prd%
    {\mathchoice{\@dsm{#1}}{\@tsm{#1}}{\@tsm{#1}}{\@tsm{#1}}\@eatprd}%
    {\mathchoice{\@dsm{#1}}{\@tsm{#1}}{\@tsm{#1}}{\@tsm{#1}}}}}
\def\@eatprd\prd{\prd@parens}
\def\sm@noparens#1{\mathchoice{\@dsm@noparens{#1}}{\@tsm{#1}}{\@tsm{#1}}{\@tsm{#1}}}
\def\lsm#1{\@ifnextchar\bgroup{\@lsm{#1}\lsm}{\@@lsm{#1}}}
\def\@lsm#1{\mathchoice{{\textstyle\sum}}{\sum}{\sum}{\sum}({\textstyle #1})\;}
\def\@@lsm#1{\mathchoice{{\textstyle\sum}}{\sum}{\sum}{\sum}({\textstyle #1}),\ }
\def\tsm#1{\@tsm{#1}\@ifnextchar\bgroup{\tsm}{}}
\def\@tsm#1{\mathchoice{{\textstyle\sum_{(#1)}}}{\sum_{(#1)}}{\sum_{(#1)}}{\sum_{(#1)}}}
\def\dsm#1{\@dsm{#1}\@ifnextchar\bgroup{\dsm}{}}
\def\@dsm#1{\sum_{#1}\,}
\def\@dsm@noparens#1{\sum_{#1}\,}
\makeatother
\makeatletter
\def\lam#1{{\lambda}\@lamarg#1:\@endlamarg\@ifnextchar\bgroup{.\,\lam}{.\,}}
\def\@lamarg#1:#2\@endlamarg{\if\relax\detokenize{#2}\relax #1\else\@lamvar{\@lameatcolon#2},#1\@endlamvar\fi}
\def\@lamvar#1,#2\@endlamvar{(#2\,{:}\,#1)}
\def\@lameatcolon#1:{#1}

\def\lamu#1{{\lambda}\@lamuarg#1:\@endlamuarg\@ifnextchar\bgroup{.\,\lamu}{.\,}}
\def\@lamuarg#1:#2\@endlamuarg{#1}
\makeatother
\newcommand{\oftype}{\mathord{:}}

\newcommand{\tminctx}[3]{#1 \vdash #2 \, \colon \, #3}

\newcommand{\tcof}{\TypeFont{t}}

\begin{document}

\title[A 2-categorical proof of Frobenius]{A 2-categorical proof of Frobenius \\ for fibrations defined from a generic point}
\author{Sina Hazratpour}

\author{Emily Riehl}

\address{Department of Mathematics\\Johns Hopkins University \\ 3400 N Charles Street \\ Baltimore, MD 21218}
\email{sinahazratpour@gmail.com} 
\email{eriehl@jhu.edu}

\date{\today}

\begin{abstract} Consider a locally cartesian closed category with an object $\II$ and a class of trivial fibrations, which admit sections and are stable under pushforward and retract as arrows. Define the fibrations to be those maps whose Leibniz exponential with the generic point of $\II$ defines a trivial fibration. Then the fibrations are also closed under pushforward.
\end{abstract}

\thanks{This work was supported by the United States Air Force Office of Scientific Research under award number FA9550-21-1-0009. Additional support was provided by the National Science Foundation via the grants DMS-1652600 and DMS-2204304, the Swedish Research Council under grant no.~2016-06596 while the second author was in residence at Institut Mittag-Leffler in Djursholm, Sweden during the Spring of 2022, and also by the National Science Foundation under Grant No.~DMS-1928930, while the second author participated in a program supported by the Mathematical Sciences Research Institute that was held in the summer of 2022 in partnership with the the Universidad Nacional Aut\'onoma de M\'exico. We thank Steve Awodey, Tim Campion, Thierry Coquand, Nicola Gambino, Maru Sarazola, Christian Sattler and Jonathan Weinberger for useful discussions during the preparation of this manuscript, and the anonymous referee for catching several typos and making various other excellent suggestions during the revision process.}

\maketitle

\tableofcontents

\section{Introduction}

Consider a locally cartesian closed category with an object $\II$, thought of as some sort of ``interval''---though in our present context no interval structure is required. Suppose further that the category comes with a class of \emph{trivial fibrations}, which satisfy a few hypotheses to be enumerated below. We define a map $p \colon \AA \to \XX$ to be a \emph{fibration} just when the induced map to the pullback in the naturality square for the evaluation transformation is a trivial fibration:
\[
  \begin{tikzcd} \AA^\II \times \II \arrow[drr, bend left, "\epsilon"] \arrow[ddr, bend right, "{p}^\II \times \II "'] \arrow[dr, dashed, "\delta\To{p}"] \\ &  \AA_{\epsilon} \arrow[d, "\epsilon^*{p}"'] \arrow[r, "{p}^*\epsilon"] \arrow[dr, phantom, "\lrcorner" very near start] & \AA \arrow[d, "{p}"] \\ & \XX^\II\times \II  \arrow[r, "\epsilon"'] & \XX
  \end{tikzcd}    
\]
This notion of fibration was first considered by Thierry Coquand \cite{C}. As noted in Proposition \ref{prop:fibration}, the map $\delta\To{p}$ coincides with the Leibniz exponential in the slice over $\II$ of ${p}$ together with the ``generic point,'' an observation made by Steve Awodey \cite{A}. Our aim in this paper is to show that these fibrations satisfy the \emph{Frobenius property}: that fibrations are closed under pushforward along other fibrations under the following hypotheses:
\begin{enumerate}
\item[(TF1)]\label{item:TF1} Trivial fibrations have sections.
\item[(TF2)]\label{item:TF2} Trivial fibrations satisfy the generalized Frobenius property: they are stable under pushforward, when considered as arrows in a slice category.\footnote{Pushforward defines a functor between slice categories. The \emph{Frobenius property} concerns stability of the objects of these categories under the pushforward functor, while the \emph{generalized Frobenius property} concerns stability of the morphisms under the pushforward functor, which implies the object version \cite[\S 6]{GS}.}
\item[(TF3)]\label{item:TF3} Trivial fibrations are stable under retract in the arrow category.
\end{enumerate}
An analogous result is proven for a different class of fibrations in \cite{GS}, which gives a more comprehensive history of this condition in the setting of (algebraic) weak factorization systems. A weak factorization system satisfies the \emph{Frobenius condition} when its right class is stable under pullback along morphisms in its left class. In a locally cartesian closed category when the fibrations define the right class of a weak factorization, the Frobenius property defined above is logically equivalent to the Frobenius condition of \cite{GS}.

The original proof of this result, due to Thierry Coquand, is written in the internal language of a locally cartesian closed category \cite{ABCFHL}. There was subsequent interest in developing a categorical account of this proof by Steve Awodey and Christian Sattler among others. They described the following proof strategy: when ${p} \colon \AA \to \XX$ is a fibration, one may construct a retract diagram for any ${q} \colon \BB \to \AA$:
\begin{equation}\label{eq:retract}
  \begin{tikzcd}
      (\Pi_\AA\BB)^\II\times\II \arrow[r, "\kappa'"] \arrow[d, "\delta \To {p}_*{q}"'] & \Pi_{\AA^\II\times\II}(\BB^\II\times\II) \arrow[d, "({p}^\II\times\II)_*(\delta\To {q})"]  \arrow[r, "\tau'"] & (\Pi_\AA\BB)^\II\times\II \arrow[d, "\delta \To {p}_*{q}"]\\ (\Pi_\AA \BB)_{\epsilon} \arrow[r, "\kappa"'] & \Pi_{\AA^\II\times\II} (\BB_{\epsilon}) \arrow[r, "\tau"'] & (\Pi_\AA \BB)_{\epsilon}
  \end{tikzcd}\end{equation}
When ${q}$ is a fibration, $\delta\To{q}$ is a trivial fibration and hence so is the middle map in this diagram, by stability under pushforward, and thus also the outer maps, by stability under retract. This is what we want to show. 

It remains to construct the retract diagram \eqref{eq:retract}, which Awodey does in \cite{A} by appealing to the universal properties of the functors involved. We adopt a more equational approach with the aim of simplifying the diagram chases necessary to verify the various commutativity conditions. In \S\ref{sec:proof}, we construct the six maps and prove that both squares commute and that both horizontal composites are identities. This proves our main result, Theorem \ref{thm:coquand}. The proofs of the various commutativity conditions are diagram chases that are greatly simplified by appealing to a theorem of Kelly and Street on the double functoriality of the mates correspondence, recalled in \S\ref{sec:mates}. This result will be applied to the adjoint triples between the slices of a locally cartesian closed category, the notation for which we introduce in \S\ref{sec:lcc}. 

Three supererogatory appendices are included to further the expository aims of this note. In \S\ref{sec:generic}, we elaborate on the claim made above, that our fibrations are defined by  Leibniz exponential in the slice over $\II$ with the generic point. In \S\ref{sec:type-theory}, we give a type-theoretic interpretation of our proof of Theorem \ref{thm:coquand} and compare this with Coquand's original proof. In the final appendix \S\ref{sec:structured}, we prove that under the natural functorial enhancements of our hypotheses on the class of trivial fibrations, that the corresponding structured fibrations admit a \emph{functorial Frobenius operator} in the sense of Gambino--Sattler \cite{GS}. Moreover, we demonstrate that the induced fibration structure on the pushforward is stable under substitution under natural conditions. The functorial formulation of the Frobenius condition is necessary to model the $\Pi$-types of homotopy type theory in a constructive meta-theory. 

While this paper was under review, Reid Barton posted a new proof of the Frobenius condition in the setting where the fibrations as defined here form the right class of a weak factorization system. By arguing  ``on the left''---that is showing that pullback along a fibration preserves the left class---instead of ``on the right'' as we do here, Barton gives a 1-categorical proof of the Frobenius condition \cite{Barton2024}.

\section{The double functoriality of the mates correspondence}\label{sec:mates}

Given a pair of adjunctions $(F \dashv U, \eta, \epsilon)$ and $(L \dashv R, \iota, \nu)$ and a pair of functors as below there is a bijective correspondence between natural transformations as displayed at the upper-left and at the lower-right implemented by pasting with the units and the counits:
\[
  \begin{tikzcd} \cA \arrow[r, "H"] \arrow[d, "F"'] \arrow[dr, phantom, "\Swarrow{\alpha}"] & \cC \arrow[d, "L"] \arrow[dr, phantom, "\mapsto"] & \arrow[dr, phantom, "\Downarrow\epsilon" very near end] & \cA \arrow[r, "H"] \arrow[d, "F" description] \arrow[dr, phantom, "\Swarrow{\alpha}"] & \cC \arrow[d, "L" description] \arrow[r, equals] & \cC \\ \cB \arrow[r, "K"'] & \cD & \cB \arrow[ur, "U"] \arrow[r, equals] & \cB \arrow[r, "K"'] & \cD \arrow[ur, "R"'] & \arrow[ul, phantom, "\Downarrow\iota" very near end]
  \end{tikzcd}
\]

\[
  \begin{tikzcd}
\cA \arrow[r, equals] \arrow[dr, "F"'] & \cA \arrow[dr, phantom, "\Searrow\beta"]\arrow[r, "H"] & \cC \arrow[dr, "L"] & \arrow[dr, phantom, "\rotatebox{180}{$\mapsto$}"] \arrow[dl, phantom, "\Downarrow\nu" very near end] & \cA \arrow[r, "H"] \arrow[dr, phantom, "\Searrow\beta"] & \cC \\ \arrow[ur, phantom, "\Downarrow\eta" very near end] & \cB \arrow[r, "K"'] \arrow[u, "U" description] & \cD \arrow[u, "R" description] \arrow[r, equals] & \cD & \cB \arrow[r, "K"'] \arrow[u, "U"] & \cD \arrow[u, "R"']
  \end{tikzcd}
\]
The corresponding pairs of 2-cells ${\alpha} \colon LH \To KF$ and $\beta \colon HU \To RK$ are called \emph{mates}.

A pasting diagram calculation proves that the mates correspondence is functorial with respect to pasting of squares in both the horizontal and vertical directions, as summarized by the following theorem:

\begin{thm}[{Kelly-Street \cite{kel74}}]\label{thm:KS} Consider the pair of double categories $\Ladj$ and $\Radj$ whose:
  \begin{itemize}
    \item objects are categories,
    \item horizontal arrows are functors,
    \item vertical arrows are fully-specified adjunctions pointing in the direction of the left adjoint, and
    \item squares of 
    \begin{itemize}
      \item $\Ladj$ are natural transformations between the squares of functors formed by the left adjoints.
      \item $\Radj$ are natural transformations between the squares of functors formed by the right adjoints.
    \end{itemize}
  \end{itemize}
  Then there is an isomorphism of double categories $\Ladj \cong \Radj$ that is the identity on objects and horizontal and vertical arrows and acts on squares by taking mates.
\end{thm}

The upshot is that a pasting equation between diagrams of squares in $\Ladj$ holds if and only if the corresponding pasting equation between diagrams of squares in $\Radj$ holds. We will use this result liberally in what follows to reduce pasting equations to simpler pasting equations.

\begin{ex}\label{ex:co-unit-mates} The units and counits of an adjunction $(F\dashv U, \eta, \epsilon)$ each arise as mates of identity transformations. By Theorem \ref{thm:KS}, the trivial pasting equality involving identity squares below-left recovers, upon taking mates, the triangle identity below-right:
  \[\begin{tikzcd}
   \cA \arrow[r, equals] \arrow[d, equals] & \cA \arrow[d, "F"] \arrow[r, "F"] & \cB \arrow[d, equals] \arrow[dr, phantom, "="] & \cA \arrow[r, "F"] \arrow[d, equals] & \cB \arrow[d, equals] & \arrow[d, phantom, "\leftrightsquigarrow"] & \cA \arrow[d, equals] \arrow[r, equals] \arrow[dr, phantom, "\Searrow\eta"] & \cA \arrow[r, "F"] \arrow[dr, phantom, "\Searrow\epsilon"] & \cB \arrow[d, equals] \arrow[dr, phantom, "="] & \cA \arrow[r, "F"] \arrow[d, equals] & \cB \arrow[d, equals]\\ \cA \arrow[r, "F"']& \cB  \arrow[r, equals] & \cB & \cA \arrow[r, "F"'] & \cB &~ & \cA \arrow[r, "F"'] & \cB \arrow[u, "U"'] \arrow[r, equals] & \cB  & \cA \arrow[r, "F"'] & \cB
  \end{tikzcd}
  \]
  The other triangle identity can be recovered similarly.
\end{ex}

\begin{rmk} Example \ref{ex:co-unit-mates} reveals that the mate of an isomorphism need not be an isomorphism. 
\end{rmk}

\begin{ex}\label{ex:conjugation-as-mates}
The \emph{conjugate} bijection between a natural transformation $\alpha \colon L \To F$ between parallel left adjoints $L \dashv R, F \dashv U$ and a natural transformation $\beta \colon U \To R$ between their right adjoints is a degenerate case of the mates correspondence in which the horizontal functors are identities.
\[
  \begin{tikzcd} \cA \arrow[r, equals] \arrow[d, "F"'] \arrow[dr, phantom, "\Swarrow\alpha"] & \cA \arrow[d, "L"] \arrow[dr, phantom, "\mapsto"] & \arrow[dr, phantom, "\Downarrow\epsilon" very near end] & \cA \arrow[r, equals] \arrow[d, "F" description] \arrow[dr, phantom, "\Swarrow\alpha"] & \cA \arrow[d, "L" description] \arrow[r, equals] & \cA \\ \cB \arrow[r, equals] & \cB & \cB \arrow[ur, "U"] \arrow[r, equals] & \cB \arrow[r, equals] & \cB \arrow[ur, "R"'] & \arrow[ul, phantom, "\Downarrow\iota" very near end]
  \end{tikzcd}
\]

\[
  \begin{tikzcd}
\cA \arrow[r, equals] \arrow[dr, "L"'] & \cA \arrow[dr, phantom, "\Searrow\beta"]\arrow[r, equals] & \cA \arrow[dr, "L"] & \arrow[dr, phantom, "\rotatebox{180}{$\mapsto$}"] \arrow[dl, phantom, "\Downarrow\nu" very near end] & \cA \arrow[r, equals] \arrow[dr, phantom, "\Searrow\beta"] & \cA \\ \arrow[ur, phantom, "\Downarrow\eta" very near end] & \cB \arrow[r, equals] \arrow[u, "U" description] & \cB \arrow[u, "R" description] \arrow[r, equals] & \cB & \cB \arrow[r, equals] \arrow[u, "U"] & \cB \arrow[u, "R"']
  \end{tikzcd}
\]
\end{ex}

\begin{rmk}\label{rmk:mate-invertibility}
While the mate of an isomorphism need not be an isomorphism the conjugate of an isomorphism is necessarily an isomorphism, as conjugate pairs can be understood as 2-cells in the vertical 2-categories contained in the double categories of Theorem \ref{thm:KS}.

In a very special case, something even stronger is true: given an adjunction $(F \dashv U, \eta, \epsilon)$ and conjugate pairs $\alpha \colon F \To F$ and $\beta \colon U \To U$ then $\alpha$ is an identity if and only if $\beta$ is an identity.
\end{rmk}

\section{Adjoint triples and locally cartesian closed categories}\label{sec:lcc}

Consider a triple of adjoint functors
\[
  \begin{tikzcd}
    \cC \arrow[r, bend left=45, "L", "\bot"'] \arrow[r, bend right=45, "R"', "\bot"] & \cD \arrow[l, "U" description]
  \end{tikzcd}
\]
The units and counits
\[\begin{tikzcd} 1_{\cC} \arrow[r, Rightarrow, "\iota"] & UL & LU \arrow[r, "\mu", Rightarrow] & 1_{\cD} & 1_{\cD} \arrow[r, "\eta", Rightarrow] & RU & UR \arrow[r, "\nu", Rightarrow] & 1_{\cC} \end{tikzcd}\]
compose to define units and counits for the composite adjunctions
\[
  \begin{tikzcd} \cC \arrow[r, bend left, "L"] \arrow[r, phantom, "\bot"] & \cD\arrow[l, bend left, "U"] \arrow[r, bend left, "U"] \arrow[r, phantom, "\bot"] & \cC \arrow[l, bend left, "R"] & \cD   \arrow[r, bend left, "U"] \arrow[r, phantom, "\bot"] & \cC \arrow[l, bend left, "R"] \arrow[r, bend left, "L"] \arrow[r, phantom, "\bot"] & \cD \arrow[l, bend left, "U"]
  \end{tikzcd}
\]
With respect to this adjoint triple:

\begin{lem}
The counit $\mu$ of $L \dashv U$  and the unit $\eta$ of $U \dashv R$ are conjugates with respect to $LU \dashv RU$ and the identity adjunction, while the unit $\iota$ of $L \dashv U$ and the counit $\nu$ of $U \dashv R$ are conjugates with respect to $UL \dashv UR$. 
\end{lem}
\begin{proof}
  The first of these identities is verified by the pasting diagram calculation, which computes the conjugate of $\mu$ by pasting with the unit of $LU \dashv RU$:
\[
  \begin{tikzcd}
    \cD \arrow[d, equals]
 \arrow[r, equals] & \cD \arrow[d, "U"] \arrow[r, equals] & \cD \arrow[d, "U"'] \arrow[r, equals] & \cD \arrow[d, equals] \arrow[dl, phantom, "\Swarrow\eta"] \\
\cD \arrow[r, "U"] \arrow[d, equals] & \cC \arrow[dl, phantom, "\Swarrow\mu"]\arrow[r, equals] \arrow[d, "L" description] & \arrow[dl, phantom, "\Swarrow\iota"]\cC \arrow[d, equals] \arrow[r, "R"'] & \cD \arrow[d, equals] \\ \cD \arrow[r, equals] & \cD \arrow[r, "U"'] & \cC \arrow[r, "R"'] & \cD \end{tikzcd} \qquad = \qquad 
\begin{tikzcd}  \cD \arrow[r, equals] \arrow[d, "U"'] & \cD \arrow[dl, phantom, "\Swarrow\eta"] \arrow[d, equals] \\ \cC \arrow[r, "R"'] & \cD
\end{tikzcd}
\]
The others are dual.
\end{proof}

In a locally cartesian closed category, every morphism ${p} \colon \AA \to \XX$ gives rise to an adjoint triple ${p}_! \dashv {p}^* \dashv {p}_*$ between the slice categories over $\AA$ and over $\XX$, which we write as 
\begin{equation}\label{eq:adjoint-triple}
  \begin{tikzcd}
    \slice{\AA} \arrow[r, bend left=45, "{p}_!", "\bot"'] \arrow[r, bend right=45, "{p}_*"', "\bot"] & \slice{\XX} \arrow[l, "{p}^*" description] 
  \end{tikzcd}
\end{equation}
so as to avoid giving an extraneous name for the locally cartesian closed category itself.

We fix the following notation for the units and counits of these adjunctions
\begin{equation}\label{eq:adjoint-triple-unit-counit-pairs} 
\begin{tikzcd}\id_{\slice{\AA}} \arrow[r, Rightarrow, "\iota"] & {p}^*{p}_! & {p}_!{p}^* \arrow[r, "\mu", Rightarrow] & \id_{\slice{\XX}} & \id_{\slice{\XX}} \arrow[r, "\eta", Rightarrow] & {p}_*{p}^* & {p}^*{p}_* \arrow[r, "\nu", Rightarrow] & \id_{\slice{\AA}} \end{tikzcd}
\end{equation}
These natural transformations compose to define units and counits for the composite adjunctions
\[
  \begin{tikzcd} \slice{\AA} \arrow[r, bend left, "{p}_!"] \arrow[r, phantom, "\bot"] & \slice{\XX}\arrow[l, bend left, "{p}^*"] \arrow[r, bend left, "{p}^*"] \arrow[r, phantom, "\bot"] & \slice{\AA} \arrow[l, bend left, "{p}_*"] & \slice{\XX}    \arrow[r, bend left, "{p}^*"] \arrow[r, phantom, "\bot"] & \slice{\AA} \arrow[l, bend left, "{p}_*"] \arrow[r, bend left, "{p}_!"] \arrow[r, phantom, "\bot"] & \slice{\XX} \arrow[l, bend left, "{p}^*"] \arrow[r, phantom, "\eqcolon"] & \slice{\XX} \arrow[r, bend left, "-\times \AA"] \arrow[r, phantom, "\bot"] & \slice{\XX} \arrow[l, bend left, "{(-)^\AA}"]
  \end{tikzcd}
\]
the latter of which defines product with and exponentiation with the object $\AA \in \slice{\XX}$.

\begin{rmk}\label{rmk:pseudofunctoriality}
  It is natural to regard the leftmost adjoints of \eqref{eq:adjoint-triple} as defining a strictly commutative functor from the locally cartesian closed category to categories, sending $\XX$ to $\slice{\XX}$ and ${p} \colon \AA \to \XX$ to ${p}_! \colon \slice{\AA} \to \slice{\XX}$. In particular, $\id_!$ is the identity functor so we may choose $\id^*$ and $\id_*$ to be identities as well.

  In general, for each morphism ${p} \colon \AA \to \XX$, we choose---once and for all---right adjoints ${p}_! \dashv {p}^* \dashv {p}_*$, choosing the identities when ${p}$ is an identity. In this way, passing to slice categories defines a normal pseudofunctor from the locally cartesian closed category to the 2-category of categories, adjoint triples, and conjugate triples of natural transformations, which is a strict functor on the leftmost adjoints.
\end{rmk}

We don't make use of the full pseudofunctoriality of the passage to adjoint triples. We do however make use to the following elementary fact that follows easily from Theorem \ref{thm:KS}.

\begin{lem}\label{lem:pseudofunctoriality}
  Any commutative rectangle in a locally cartesian closed category
  \[
    \begin{tikzcd} \CC \arrow[r, "r"] \arrow[d, "t"'] & \BB \arrow[d, "s"] \arrow[r, "{q}"] & \AA \arrow[d, "{p}"] \\ \ZZ \arrow[r, "y"'] & \YY \arrow[r, "x"'] & \XX
    \end{tikzcd}
  \]
  gives rise to pasting equations between the canonical natural isomorphisms:
  \[
  \begin{tikzcd} \slice{\CC} \arrow[from=r, "r^*"'] \arrow[from=d, "t^*"] \arrow[dr, phantom, "\cong"] & \slice{\BB} \arrow[from=d, "s^*"'] \arrow[from=r, "{q}^*"'] \arrow[dr, phantom, "\cong"] & \slice{\AA} \arrow[from=d, "{p}^*"'] \arrow[dr, phantom, "="] & \slice{\CC} \arrow[from=d, "t^*"] \arrow[from=r, "r^*{q}^*"'] \arrow[dr, phantom, "\cong"] & \slice{\AA} \arrow[from=d, "{p}^*"'] &  \slice{\CC} \arrow[dr, phantom, "\cong"] \arrow[r, "r_*"] \arrow[d, "t_*"'] & \arrow[dr, phantom, "\cong"]\slice{\BB} \arrow[d, "s_*"] \arrow[r, "{q}_*"] & \slice{\AA} \arrow[d, "{p}_*"] \arrow[dr, phantom, "="] & \slice{\CC} \arrow[d, "t_*"'] \arrow[r, "{q}_*r_*"] \arrow[dr, phantom, "\cong"] & \slice{\AA} \arrow[d, "{p}_*"]\\ \slice{\ZZ} \arrow[from=r, "y^*"] & \slice{\YY} \arrow[from=r, "x^*"] & \slice{\XX} & \slice{\ZZ} \arrow[from=r, "y^*x^*"] & \slice{\XX} & \slice{\ZZ} \arrow[r, "y_*"'] & \slice{\YY} \arrow[r, "x_*"'] & \slice{\XX} & \slice{\ZZ} \arrow[r, "x_*y_*"'] & \slice{\XX}
  \end{tikzcd}
\]
\end{lem}

\begin{proof}
Starting from the pasting equation between the identity natural transformations:
\[
  \begin{tikzcd} \slice{\CC} \arrow[r, "r_!"] \arrow[d, "t_!"'] & \slice{\BB} \arrow[d, "s_!"] \arrow[r, "{q}_!"] & \slice{\AA} \arrow[d, "{p}_!"] \arrow[dr, phantom, "="] & \slice{\CC} \arrow[d, "t_!"'] \arrow[r, "{q}_!r_!"] & \slice{\AA} \arrow[d, "{p}_!"]\\ \slice{\ZZ} \arrow[r, "y_!"'] & \slice{\YY} \arrow[r, "x_!"'] & \slice{\XX} & \slice{\ZZ} \arrow[r, "x_!y_!"'] & \slice{\XX}
  \end{tikzcd}
\]
we may take mates first and the vertical direction and then in the horizontal direction to obtain a pasting equation between the conjugate isomorphisms:
\[
  \begin{tikzcd} \slice{\CC} \arrow[from=r, "r^*"'] \arrow[from=d, "t^*"] \arrow[dr, phantom, "\cong"] & \slice{\BB} \arrow[from=d, "s^*"'] \arrow[from=r, "{q}^*"'] \arrow[dr, phantom, "\cong"] & \slice{\AA} \arrow[from=d, "{p}^*"'] \arrow[dr, phantom, "="] & \slice{\CC} \arrow[from=d, "t^*"] \arrow[from=r, "r^*{q}^*"'] \arrow[dr, phantom, "\cong"] & \slice{\AA} \arrow[from=d, "{p}^*"']\\ \slice{\ZZ} \arrow[from=r, "y^*"] & \slice{\YY} \arrow[from=r, "x^*"] & \slice{\XX} & \slice{\ZZ} \arrow[from=r, "y^*x^*"] & \slice{\XX}
  \end{tikzcd}
\]
Repeating this procedure yields the corresponding pasting equation for the pushforward functors.
\end{proof}

\begin{war}\label{war:pseudofunctoriality}
  While ${q}_!r_! = ({q}r)_!$, the conjugate of this identity transformation typically defines non-identity isomorphisms $r^*{q}^* \cong ({q}r)^*$ and ${q}_*r_* \cong ({q}r)_*$, which form a key component of the pseudofunctor described in Remark \ref{rmk:pseudofunctoriality}. One could obtain analogous pasting equations in which the right-hand squares involved the functors $({q}r)^*$ and $({q}r)_*$ at the cost of composing the left-hand pasted rectangles with these isomorphisms.
\end{war}

\begin{lem}\label{lem:beck-chevalley-isos} For any pullback square in a locally cartesian closed category
  \[
    \begin{tikzcd}  \BB \arrow[d, "s"'] \arrow[dr, phantom, "\lrcorner" very near start] \arrow[r, "{q}"] & \AA \arrow[d, "{p}"] \\  \YY \arrow[r, "x"'] & \XX
    \end{tikzcd}
  \]
  all of the mates of the identity transformation $p_! q_! = x_! s_!$ are isomorphisms including in particular
  \begin{equation}\label{eq:beck-chevalley-isos}
  \begin{tikzcd} \slice{\BB} \arrow[r, "q_!"] & \slice{\AA} & & \slice{\BB} \arrow[d, "s_*"'] \arrow[dr, phantom, "\cong\Nwarrow"] & \slice{\AA} \arrow[l, "q^*"'] \arrow[d, "p_*"] \\ \slice{\YY} \arrow[u, "s^*"] \arrow[r, "x_!"'] \arrow[ur, phantom, "\cong\Searrow"] & \slice{\XX} \arrow[u, "p^*"'] & & \slice{\YY} & \slice{\XX} \arrow[l, "x^*"]
  \end{tikzcd}
  \end{equation}
\end{lem}

The statement refers to all natural transformations obtained by iteratively applying the mates correspondence, where applicable. For instance, after taking mates twice, the identity transformation $p_!q_! = x_! s_!$ is carried to its conjugate transformation $s^*x^* \cong q^* p^*$, which is an isomorphism by Remark \ref{rmk:mate-invertibility}. The iterated mate $p_*q_* \cong x_* s_*$ is shown to be invertible similarly.  The invertibility of the mates on display in \eqref{eq:beck-chevalley-isos} is more delicate, requiring the pullback condition. We refer to these as the \textbf{Beck-Chevalley isomorphisms} \cite[\S 14]{Pavlovic}.

\begin{proof}
The left-hand and right-hand mates of \eqref{eq:beck-chevalley-isos} are conjugates. Thus by Remark \ref{rmk:mate-invertibility} it suffices to prove the invertibility of the former, which we do by arguing that its component at any $y \colon \ZZ \to \YY$ is an isomorphism. This component is given by the canonical map from the pullback of $y$ over $s$ composed with $q$ to the pullback of $xy$ along $p$. 
\[
\begin{tikzcd}
  \ZZ_p \arrow[drrr, bend left=14, "p^*(xy)"] \arrow[from=dr, dashed, "\cong"'] \arrow[ddr, bend right] \arrow[ddr, phantom, "\lrcorner" very near start] 
 \\ [-1em] & [-1em] \ZZ_s \arrow[r, "s^*y"] \arrow[d] \arrow[dr, phantom, "\lrcorner" very near start] & \BB \arrow[d, "s"'] \arrow[dr, phantom, "\lrcorner" very near start] \arrow[r, "q"] & \AA \arrow[d, "p"] \\
& \ZZ \arrow[r, "y"'] & \YY \arrow[r, "x"'] & \XX
\end{tikzcd}
\]
Since the upper right-hand square is a pullback, both rectangles are pullbacks, proving that this comparison map is invertible as desired.
\end{proof}

\begin{rmk}\label{rmk:lccc}
Now consider a morphism
\[
  \begin{tikzcd} \BB \arrow[dr, "{p}{q}"'] \arrow[rr, "{q}"] & & \AA \arrow[dl, "{p}"] \\ & \XX
  \end{tikzcd}
\]
in $\slice{\XX}$. The morphism ${q}$ induces a conjugate pair of natural transformations 
\[
  \begin{tikzcd} \slice{\XX} \arrow[r, bend left, "-\times \BB"] \arrow[r, bend right, "-\times \AA"'] \arrow[r, phantom, "\Downarrow{q}"] & \slice{\XX} & \slice{\XX} \arrow[r, bend left, "{(-)^\BB}"] \arrow[r, bend right, "{(-)^\AA}"'] \arrow[r, phantom, "\Uparrow{q}"] & \slice{\XX}
  \end{tikzcd}
  \]
  defined as follows:
\[   \begin{tikzcd}
   \slice{\BB} \arrow[dr, equals]  \arrow[rrr, "{({p}{q})_!}"] & & & \slice{\XX} \arrow[d, equals] & 
   & & \slice{\XX} \arrow[r, equals] & \slice{\XX}
   \\ & \slice{\BB} \arrow[r, "{q}_!"] & \slice{\AA} \arrow[r, "{p}_!"] & \slice{\XX} &  
   & \slice{\AA} \arrow[r, equals] \arrow[dr, phantom, "\Searrow\eta"] & \slice{\AA} \arrow[u, "{p}_*"] \arrow[r, phantom, "\cong"] & ~  \\  \arrow[r, phantom, "\cong"] & \slice{\AA} \arrow[u, "{q}^*"] \arrow[r, equals] & \slice{\AA} \arrow[ul, phantom, "\Searrow\mu"] \arrow[u, equals] & & 
   \slice{\XX} \arrow[r, "{p}^*"] & \slice{\AA} \arrow[u, equals] \arrow[r, "{q}^*"'] & \slice{\BB} \arrow[u, "{q}_*"'] \arrow[dr, equals] \\ \slice{\XX} \arrow[uuu, "{({p}{q})^*}"] \arrow[r, equals] & \slice{\XX} \arrow[u, "{p}^*"']  & & & \slice{\XX} \arrow[u, equals] \arrow[rrr, "{({p}{q})^*}"'] & \arrow[u, phantom, "\cong"] & & \slice{\BB} \arrow[uuu, "{({p}{q})_*}"']
  \end{tikzcd}
\]
That is, these transformations are defined as whiskered composites of the counit $\mu$ and unit $\eta$ of ${q}_! \dashv {q}^*$ and ${q}^* \dashv {q}_*$, respectively, padded by the canonical isomorphisms between the functors $({p}{q})^* \cong {q}^*{p}^*$ and ${p}_*{q}_* \cong ({p}{q})_*$. Note also that these isomorphisms are also conjugates, with the isomorphism ${q}^* {p}^*\cong ({p}{q})^*$ also the conjugate of the identity $({p}{q})_! = {p}_!{q}_!$.
\end{rmk}

\section{The proof}\label{sec:proof}

Consider a locally cartesian closed category with a class of trivial fibrations. Pick an object $\II$ which we think of as an ``interval.'' For any $\XX$ we have natural maps 
\begin{equation}\label{eq:varpi-epsilon-components}
  \begin{tikzcd} \XX^\II & \XX^\II \times \II \arrow[l, "\varpi"'] \arrow[r, "\epsilon"] & \XX\end{tikzcd}
\end{equation}
called \textbf{projection} and \textbf{evaluation} that are defined as components of the following composite natural transformations associated to the map $! \colon \II \to \ast$ from \eqref{eq:adjoint-triple-unit-counit-pairs}.

\begin{equation}\label{eq:varpi-epsilon} \varpi \coloneq
\begin{tikzcd} &  & \slice{\II} \arrow[r, "!_!"] \arrow[dr, phantom, "\Searrow\mu"] & \slice{\ast} \arrow[d, equals] \\ \slice{\ast} \arrow[r, "!^*"] & \slice{\II}  \arrow[r, "!_*"] & \slice{\ast} \arrow[u, "!^*"] \arrow[r, equals] & \slice{\ast}
\end{tikzcd} \qquad
 \qquad \epsilon \coloneq
  \begin{tikzcd} \slice{\ast} \arrow[r, "!^*"] \arrow[dr, phantom, "\Searrow\nu"] & \slice{\II} \arrow[r, "!_!"] \arrow[d, equals]& \slice{\ast} \arrow[d, equals] \\ \slice{\II} \arrow[u, "!_*"] \arrow[r, equals] & \slice{\II} \arrow[r, "!_!"] \arrow[dr, phantom, "\Searrow\mu"] & \slice{\ast} \arrow[d, equals] \\ \slice{\ast} \arrow[u, "!^*"] \arrow[r, equals] & \slice{\ast} \arrow[u, "!^*"] \arrow[r, equals] & \slice{\ast}
  \end{tikzcd}
\end{equation}

\begin{defn}\label{defn:fibration} A map ${p} \colon \AA \to \XX$ is a \emph{fibration} just when the induced map $\delta\To{p}$ to the pullback in the naturality square for the evaluation transformation is a trivial fibration:
\begin{equation}\label{eq:fibration}
  \begin{tikzcd} \AA^\II \times \II \arrow[drr, bend left, "\epsilon"] \arrow[ddr, bend right, "{p}^\II \times \II "'] \arrow[dr, dashed, "\delta\To{p}"] \\ &  \AA_{\epsilon} \arrow[d, "\epsilon^*{p}"'] \arrow[r, "{p}^*\epsilon"] \arrow[dr, phantom, "\lrcorner" very near start] & \AA \arrow[d, "{p}"] \\ & \XX^\II\times \II  \arrow[r, "\epsilon"'] & \XX
  \end{tikzcd}    
\end{equation}
The map $\delta\To{p}$ encodes ``evaluation at a generic 
point'' in $\II$ in a sense elaborated upon in Appendix \ref{sec:generic}. 
\end{defn}

Our aim is to prove the following theorem.

\begin{thm}[Coquand]\label{thm:coquand}
  Consider a locally cartesian closed category with an object $\II$ and a class of trivial fibrations, which admit sections and are stable under pushforward and retract. Then the fibrations are also closed under pushforward.
\end{thm}

Our proof follows the outline described in the introduction. It remains only to construct the retract diagram \eqref{eq:retract}, which we achieve over a series of lemmas.

\begin{lem}
  The component of the whiskered counit 
  \begin{equation}\label{eq:leibniz-map}\begin{tikzcd} &  \slice{\XX^\II} \arrow[r, "\varpi^*"] \arrow[dr, phantom, "\Downarrow\nu"] & \slice{\XX^\II \times \II} \arrow[d, equals]\\ 
    \slice{\XX} \arrow[r, "\epsilon^*"] & \slice{\XX^\II \times \II} \arrow[r, equals] \arrow[u, "\varpi_*"]  & \slice{\XX^\II \times \II}
  \end{tikzcd}
\end{equation}
  at ${p} \colon \AA \to \XX$ is the map $\delta \To {p}$.
\end{lem}
\begin{proof}
  The domain $\varpi^*\varpi_*\epsilon^*$ of the natural transformation \eqref{eq:leibniz-map} is right adjoint to the functor $\epsilon_!\varpi^*\varpi_!$. The composite left adjoint $\epsilon_!\varpi^*$ sends $x \colon \ZZ \to \XX^\II$ to its transpose $\epsilon \cdot (x \times \II) \colon \ZZ \times \II \to \XX$, so the right adjoint $\varpi_*\epsilon^*$ sends ${p} \colon \AA \to \XX$ to ${p}^\II \colon \AA^\II\to \XX^\II$. Thus, the functor $\varpi^*\varpi_*\epsilon^*$ sends $p$ to $p^{\II} \times \II$.

  This identifies the whiskered natural transformation $\nu\epsilon^*$ as a map from $\AA^\II \times \II$ to $\AA_\epsilon$ over $\XX^\II \times \II$. By the universal property of the pullback that defines $\AA_\epsilon$, to identify this map it suffices to consider the pasted composite with $\mu \colon \epsilon_! \epsilon^* \To \id$ since this pasting corresponds to composing with the pullback square and considering the resulting map $\AA^\II \times \II \to \AA$ over $\XX$. But then the pasted composite in question

\[  \begin{tikzcd} \slice{\XX^\II} \arrow[r, "\varpi^*"] \arrow[dr, phantom, "\Searrow\nu"] & \slice{\XX^\II\times\II} \arrow[r, "\epsilon_!"] \arrow[d, equals]& \slice{\XX} \arrow[d, equals] \\ \slice{\XX^\II\times\II} \arrow[u, "\varpi_*"] \arrow[r, equals] & \slice{\XX^\II \times\II} \arrow[r, "\epsilon_!"] \arrow[dr, phantom, "\Searrow\mu"] & \slice{\XX} \arrow[d, equals] \\ \slice{\XX} \arrow[u, "\epsilon^*"] \arrow[r, equals] & \slice{\XX} \arrow[u, "\epsilon^*"] \arrow[r, equals] & \slice{\XX}
  \end{tikzcd}\]
  is the counit for the composite adjunction $\epsilon_!\varpi^* \dashv \varpi_*\epsilon^*$. The left adjoint of this adjunction sends $x \colon \ZZ \to \XX^\II$ to its transpose $\epsilon \cdot (x \times \II) \colon \ZZ \times \II \to \XX$ while the right adjoint sends ${p} \colon \AA \to \XX$ to ${p}^\II \colon \AA^\II\to \XX^\II$ so the counit map is the map $\epsilon \colon \AA^\II \times \II \to \AA$ over $\XX$ as claimed.
\end{proof}

The commutative square below-left, induces a commutative square between the left adjoints below-center. 
\[
  \begin{tikzcd}
\AA^\II \times \II \arrow[d, "{p}^\II \times \II"'] \arrow[r, "\epsilon"] & \AA \arrow[d, "{p}"] \arrow[drr, phantom, "\rightsquigarrow"]  & & \slice{\AA^\II \times \II} \arrow[d, "({p}^\II \times \II)_!"'] \arrow[r, "\epsilon_!"] & \slice{\AA} \arrow[d, "{p}_!"]  \arrow[drr, phantom, "\rightsquigarrow"] & &  \slice{\AA^\II \times \II} \arrow[d, "({p}^\II \times \II)_*"'] \arrow[from=r, "\epsilon^*"'] \arrow[dr, phantom, "\Uparrow\kappa"] & \slice{\AA} \arrow[d, "{p}_*"]   \\ \XX^\II \times \II \arrow[r, "\epsilon"'] & \XX   & & \slice{\XX^\II \times \II} \arrow[r, "\epsilon_!"'] & \slice{\XX}    && \slice{\XX^\II \times \II} \arrow[from=r, "\epsilon^*"] & \slice{\XX}    
  \end{tikzcd}
\]
After taking mates first in the vertical direction, then in the horizontal direction, and then in the vertical direction again, we obtain a canonical natural transformation displayed above-right whose component at ${q} \colon \BB \to \AA$ defines a map $\kappa \colon (\Pi_\AA\BB)_\epsilon \to \Pi_{\AA^\II \times \II} (\BB_\epsilon)$.

Starting from the pullback square
\[
  \begin{tikzcd} \AA^\II \times \II \arrow[r, "\varpi"] \arrow[d, "{p}^\II \times \II"'] \arrow[dr, phantom, "\lrcorner" very near start] & \AA^\II \arrow[d, "{p}^\II"] \\ \XX^\II \times \II \arrow[r, "\varpi"'] & \XX^\II
  \end{tikzcd}
\]
by Lemma \ref{lem:beck-chevalley-isos}, 
a similar process---taking successive mates of identity transformations---defines a pair of isomorphisms displayed below:
\begin{equation}\label{eq:canonical-isos}
  \begin{tikzcd} \slice{\AA^\II \times \II} \arrow[d, "{({p}^\II \times \II)_*}"'] \arrow[dr, phantom, "\cong"] & \slice{\AA^\II} \arrow[dr, phantom, "\cong"] \arrow[d, "{p}^\II_*" description] \arrow[l, "\varpi^*"']& \arrow[l, "\varpi_*"'] \slice{\AA^\II \times \II} \arrow[d, "{({p}^\II\times\II)_*}"] \\ \slice{\XX^\II \times \II} & \slice{\XX^\II}\arrow[l, "\varpi^*"] & \arrow[l, "\varpi_*"] \slice{\XX^\II \times \II} 
  \end{tikzcd}
\end{equation}

\begin{lem}\label{lem:canonical-counit-iso}
  The pasted composites are equal
  \[
    \begin{tikzcd} 
      \slice{\AA^\II \times \II} \arrow[d, equals] \arrow[rr, equals] & \arrow[d, phantom, "\Uparrow\nu"] & \slice{\AA^\II \times \II} \arrow[d, equals] & \slice{\AA^\II \times \II} \arrow[rr, equals] \arrow[d, "({p}^\II \times \II)_*"'] & & \slice{\AA^\II \times \II} \arrow[d, "{({p}^\II \times \II)_*}"] \\ \slice{\AA^\II \times \II} \arrow[d, "{({p}^\II \times \II)_*}"'] \arrow[dr, phantom, "\cong"] & \slice{\AA^\II} \arrow[dr, phantom, "\cong"] \arrow[d, "{p}^\II_*" description] \arrow[l, "\varpi^*"']& \arrow[l, "\varpi_*"'] \slice{\AA^\II \times \II} \arrow[d, "{({p}^\II\times\II)_*}"] \arrow[r, phantom, "="] & \slice{\XX^\II \times \II} \arrow[d, equals] \arrow[rr, equals] & \arrow[d, phantom, "\Uparrow\nu"] & \slice{\XX^\II \times \II} \arrow[d, equals] \\ \slice{\XX^\II \times \II} & \slice{\XX^\II}\arrow[l, "\varpi^*"] & \arrow[l, "\varpi_*"] \slice{\XX^\II \times \II} & \slice{\XX^\II \times \II} & \slice{\XX^\II}\arrow[l, "\varpi^*"] & \arrow[l, "\varpi_*"] \slice{\XX^\II \times \II} 
    \end{tikzcd}
  \]
\end{lem}
\begin{proof}
  Taking mates with respect to the vertical adjunctions, then the horizontal adjunctions, and then the vertical adjunctions again yields the pasting equation:
\[
  \begin{tikzcd} \slice{\AA^\II \times \II} \arrow[d, "{({p}^\II \times \II)_!}"']  & \slice{\AA^\II} 
  \arrow[dr, phantom, "\cong"] 
  \arrow[d, "{p}^\II_!" description] \arrow[from=l, "\varpi_!"]& \arrow[from=l, "\varpi^*"] \slice{\AA^\II \times \II} \arrow[d, "{({p}^\II\times\II)_!}" description]  & \slice{\AA^\II \times \II} & \slice{\AA^\II} \arrow[from=l, "\varpi_!"] & \slice{\AA^\II \times \II} \arrow[from=l, "\varpi^*"]  \\ \slice{\XX^\II \times \II} & \slice{\XX^\II}\arrow[from=l, "\varpi_!"'] & \arrow[from=l, "\varpi^*"'] \slice{\XX^\II \times \II} \arrow[r, phantom, "\overset{?}{=}"] & 
    \slice{\AA^\II \times \II} \arrow[d, "({p}^\II \times\II)_!"'] \arrow[u,equals] \arrow[rr, equals] & \arrow[u, phantom, "\Uparrow\iota"] &\arrow[u, equals] \slice{\AA^\II \times \II} \arrow[d, "({p}^\II \times\II)_!"'] \\
    \slice{\XX^\II \times \II} \arrow[u,equals] \arrow[rr, equals] & \arrow[u, phantom, "\Uparrow\iota"] &\arrow[u, equals] \slice{\XX^\II \times \II} & 
    \slice{\XX^\II \times \II} \arrow[rr, equals] & & \slice{\XX^\II \times \II} 
  \end{tikzcd}
\]
which holds if and only if 
\[
  \begin{tikzcd} \slice{\AA^\II \times \II} \arrow[d, equals] \arrow[r, "{({p}^\II \times \II)_!}"] & \slice{\XX^\II \times \II} \arrow[d,equals] \arrow[r, "\varpi_!"] \arrow[dr, phantom, "\Uparrow\iota"] & \slice{\XX^\II} \arrow[d, "\varpi^*"] \arrow[dr,phantom, "\overset{?}{=}"] & \slice{\AA^\II \times \II} \arrow[d, equals] \arrow[r, "\varpi_!"] \arrow[dr, phantom, "\Uparrow\iota"] & \slice{\AA^\II} \arrow[d, "\varpi^*"] \arrow[r, "{{p}^\II_!}"] \arrow[dr, phantom, "\cong"] & \slice{\XX^\II} \arrow[d, "\varpi^*"] \\
    \slice{\AA^\II \times \II} \arrow[r, "{({p}^\II \times \II)_!}"'] & \slice{\XX^\II \times \II} \arrow[r, equals] & \slice{\XX^\II \times \II} & \slice{\AA^\II \times \II} \arrow[r, equals] & \slice{\AA^\II \times \II} \arrow[r, "{({p}^\II \times \II)_!}"'] & \slice{\XX^\II \times \II}
  \end{tikzcd}
\]
This is because we can post-compose the first pasting equality with the inverse of the displayed Beck-Chevalley isomorphism  to get the second pasting equality, and vice versa. Upon taking mates in the vertical direction this reduces to a pasting equation between identities, which of course holds.
\end{proof}

The natural isomorphism \eqref{eq:canonical-isos} composes with $\kappa$ to define a natural transformation
\begin{equation}\label{eq:top-section}
  \begin{tikzcd} \slice{\AA^\II \times \II} \arrow[d, "{({p}^\II \times \II)_*}"'] \arrow[dr, phantom, "\cong"] & \slice{\AA^\II} \arrow[dr, phantom, "\cong"] \arrow[d, "{p}^\II_*" description] \arrow[l, "\varpi^*"']& \arrow[l, "\varpi_*"'] \slice{\AA^\II \times \II} \arrow[d, "{({p}^\II\times\II)_*}" description] \arrow[dr, phantom, "\Uparrow\kappa"] & \arrow[l, "\epsilon^*"']\slice{\AA} \arrow[d, "{p}_*"] \\ \slice{\XX^\II \times \II} & \slice{\XX^\II}\arrow[l, "\varpi^*"] & \arrow[l, "\varpi_*"] \slice{\XX^\II \times \II} & \slice{\XX} \arrow[l, "\epsilon^*"]
  \end{tikzcd}
\end{equation} whose component 
 at
${q} \colon \BB \to \AA$ defines a map $\kappa' \colon (\Pi_\AA\BB)^\II \times \II \to \Pi_{\AA^\II \times \II} ( \BB^\II \times \II)$.

\begin{lem}\label{lem:section-square}
  For any ${p} \colon \AA \to \XX$ and ${q} \colon \BB \to \AA$, the square of maps in $\slice{\XX^\II \times \II}$ commutes:
  \[
    \begin{tikzcd}
        (\Pi_\AA\BB)^\II\times\II \arrow[r, "\kappa'"] \arrow[d, "\delta \To {p}_*{q}"'] & \Pi_{\AA^\II\times\II}\BB^\II\times\II \arrow[d, "({p}^\II\times\II)_*(\delta\To {q})"]  \\ (\Pi_\AA \BB)_{\epsilon} \arrow[r, "\kappa"'] & \Pi_{\AA^\II\times\II} (\BB_{\epsilon}) 
    \end{tikzcd}\]
\end{lem}
\begin{proof}
  Here the top-right natural transformation in the statement is the left pasted composite, while the left-bottom natural transformation is the right pasted composite:
  \[
  \begin{tikzcd}
    \slice{\AA^\II \times \II} \arrow[d, equals] \arrow[rr, equals] & \arrow[d, phantom, "\Uparrow\nu"] & \slice{\AA^\II \times \II} \arrow[d, equals] & \slice{\AA} \arrow[d, equals]\arrow[l, "\epsilon^*"']  &  \slice{\AA^\II \times \II}\arrow[d, "({p}^\II \times \II)_*"'] \arrow[rr, equals] & & \slice{\AA^\II \times \II} \arrow[d, "({p}^\II \times \II)_*"'] \arrow[dr, phantom, "\Uparrow\kappa"] & \slice{\AA} \arrow[l, "\epsilon^*"'] \arrow[d, "{p}_*"]  \\ \slice{\AA^\II \times \II} \arrow[d, "{({p}^\II \times \II)_*}"'] \arrow[dr, phantom, "\cong"] & \slice{\AA^\II} \arrow[dr, phantom, "\cong"] \arrow[d, "{p}^\II_*" description] \arrow[l, "\varpi^*"']& \arrow[l, "\varpi_*"'] \slice{\AA^\II \times \II} \arrow[d, "{({p}^\II\times\II)_*}" description] \arrow[dr, phantom, "\Uparrow\kappa"] & \arrow[l, "\epsilon^*"']\slice{\AA} \arrow[d, "{p}_*"] \arrow[r, phantom, "\overset{?}{=}"] & \slice{\XX^\II \times \II} \arrow[rr, equals] \arrow[d, equals] & \arrow[d, phantom, "\Uparrow\nu"] & \slice{\XX^\II \times \II} \arrow[d, equals] & \slice{\XX} \arrow[l, "\epsilon^*"]  \arrow[d, equals]\\ \slice{\XX^\II \times \II} & \slice{\XX^\II}\arrow[l, "\varpi^*"] & \arrow[l, "\varpi_*"] \slice{\XX^\II \times \II} & \slice{\XX} \arrow[l, "\epsilon^*"] & \slice{\XX^\II \times \II} & \slice{\XX^\II} \arrow[l, "\varpi^*"] & \slice{\XX^\II \times \II} \arrow[l, "\varpi_*"] & \slice{\XX} \arrow[l, "\epsilon^*"]
  \end{tikzcd}
\]
This follows immediately from Lemma \ref{lem:canonical-counit-iso}.
\end{proof}

It remains to construct the right-hand square of the retract diagram \eqref{eq:retract}. This is where we use the hypothesis that ${p} \colon \AA \to \XX$ is a fibration.

\begin{lem}\label{lem:section-retraction} A section $s$ to $\delta\To{p}$ induces a natural transformation
  \[
    \begin{tikzcd} \slice{\AA^\II \times \II} \arrow[d, "{({p}^\II\times\II)_*}"'] \arrow[dr, phantom, "\Downarrow\tau"] & \arrow[l, "\epsilon^*"']\slice{\AA} \arrow[d, "{p}_*"] \\ \slice{\XX^\II \times \II} & \slice{\XX} \arrow[l, "\epsilon^*"]
    \end{tikzcd}\]
that defines a retraction to $\kappa \colon \epsilon^*{p}_* \To ({p}^\II \times\II)_*\epsilon^*$. In particular, such a natural transformation exists when ${p} \colon \AA \to \XX$ is a fibration.
\end{lem}
\begin{proof}
Any section $s$ to $\delta \To{p}$  defines a commutative diagram
\begin{equation}\label{eq:section}
  \begin{tikzcd} \AA^\II \times \II \arrow[drr, bend left, "\epsilon"] \arrow[ddr, bend right, "{p}^\II \times \II "'] \arrow[dr, dashed, "\delta\To{p}"] \\ & \arrow[ul, bend left, dotted, "s"] \AA_{\epsilon} \arrow[d, "\epsilon^*{p}"'] \arrow[r, "{p}^*\epsilon"] \arrow[dr, phantom, "\lrcorner" very near start] & \AA \arrow[d, "{p}"] \\ & \XX^\II\times \II  \arrow[r, "\epsilon"'] & \XX
  \end{tikzcd}    
\end{equation}
On account of the commutative diagram \eqref{eq:section}, the natural isomorphism between pullbacks factors as follows:
\[
\begin{tikzcd} \arrow[dr, phantom, "\cong"] \slice{\AA_{\epsilon}} & \slice{\AA} \arrow[l, "({p}^*\epsilon)^*"']\\ \slice{\XX^\II\times\II} \arrow[u, "{(\epsilon^*{p})^*}"] & \slice{\XX} \arrow[l, "\epsilon^*"] \arrow[u, "{p}^*"']
\end{tikzcd}  \quad = \quad 
\begin{tikzcd} \slice{\AA_\epsilon} \arrow[r, equals] & \slice{\AA_{\epsilon}}\arrow[dr, phantom, "\cong"] & \slice{\AA} \arrow[l, "({p}^*\epsilon)^*"']\\ \arrow[r, phantom, "\cong"] & \slice{\AA^\II\times\II} \arrow[dr, phantom, "\cong"]\arrow[u, "s^*"] &
    \slice{\AA} \arrow[u, equals] \arrow[l, "\epsilon^*"]
    \\ \slice{\XX^\II \times \II} \arrow[uu, "{(\epsilon^*{p})^*}"] \arrow[r, equals] & \slice{\XX^\II\times\II} \arrow[u, "{({p}^\II\times\II)^*}"] & \slice{\XX} \arrow[l, "\epsilon^*"] \arrow[u, "{p}^*"']
\end{tikzcd} 
\]
by Lemma \ref{lem:pseudofunctoriality} and Warning \ref{war:pseudofunctoriality}.
This gives rise to a pasting equation between the mates
\[
\begin{tikzcd}  \slice{\AA_{\epsilon}} \arrow[dr, phantom, "\cong"] \arrow[d, "{(\epsilon^*{p})_*}"'] & \slice{\AA} \arrow[l, "({p}^*\epsilon)^*"'] \arrow[d, "{p}_*"]\\ \slice{\XX^\II\times\II}  & \slice{\XX} \arrow[l, "\epsilon^*"]
\end{tikzcd}  \quad = \quad 
\begin{tikzcd}   \slice{\AA_{\epsilon}} \arrow[r, equals] \arrow[dd, "(\epsilon^*{p})_*"'] &  \slice{\AA_{\epsilon}} \arrow[d, "s_*"']  \arrow[dr,phantom, "\Uparrow\sigma"] & \slice{\AA} \arrow[l, "({p}^*\epsilon)^*"']\\ \arrow[r, phantom, "\cong"] & \slice{\AA^\II\times\II} \arrow[d, "{({p}^\II\times\II)_*}"'] \arrow[dr, phantom, "\Uparrow\kappa"] &
    \slice{\AA} \arrow[u, equals] \arrow[l, "\epsilon^*"] \arrow[d, "{p}_*"]
    \\\slice{\XX^\II\times\II}  \arrow[r, equals] & \slice{\XX^\II\times\II} & \slice{\XX} \arrow[l, "\epsilon^*"] 
\end{tikzcd} 
\]
where the unlabeled natural transformation in the left-hand square is the Beck-Chevalley isomorphism associated to the pullback. Composing $\sigma$ with the pair of unlabeled isomorphisms in the diagram above, we obtain a retract $\tau \colon \Pi_{\AA^\II\times\II}(\BB_{\epsilon}) \to (\Pi_\AA\BB)_{\epsilon}$ to $\kappa \colon (\Pi_\AA\BB)_{\epsilon} \to \Pi_{\AA^\II\times\II}(\BB_{\epsilon})$.
\end{proof}

Consider the pasted composite:
\begin{equation}\label{eq:top-retract}
  \begin{tikzcd} \slice{\AA^\II \times \II} \arrow[d, "{({p}^\II \times \II)_*}"'] \arrow[dr, phantom, "\cong"] & \slice{\AA^\II} \arrow[dr, phantom, "\cong"] \arrow[d, "{p}^\II_*" description] \arrow[l, "\varpi^*"']& \arrow[l, "\varpi_*"'] \slice{\AA^\II \times \II} \arrow[d, "{({p}^\II\times\II)_*}" description] \arrow[dr, phantom, "\Downarrow\tau"] & \arrow[l, "\epsilon^*"']\slice{\AA} \arrow[d, "{p}_*"] \\ \slice{\XX^\II \times \II} & \slice{\XX^\II}\arrow[l, "\varpi^*"] & \arrow[l, "\varpi_*"] \slice{\XX^\II \times \II} & \slice{\XX} \arrow[l, "\epsilon^*"]
  \end{tikzcd}
\end{equation}
whose unlabeled isomorphisms are the inverses of the isomorphisms \eqref{eq:canonical-isos}.
 This defines the final natural transformation $\tau' \colon \Pi_{\AA^\II \times \II}(\BB^\II \times \II) \To (\Pi_\AA\BB)^\II \times \II$, appearing in the diagram \eqref{eq:retract}. The proof of Theorem \ref{thm:coquand} is completed by the following proposition.

\begin{prop} The maps that we have constructed for any map ${q} \colon \BB \to \AA$ and any fibration ${p} \colon \AA \to \XX$ assemble into a retract diagram:
\[\begin{tikzcd}
  (\Pi_\AA\BB)^\II\times\II \arrow[r, "\kappa'"] \arrow[d, "\delta \To {p}_*{q}"'] & \Pi_{\AA^\II\times\II}(\BB^\II\times\II) \arrow[d, "({p}^\II\times\II)_*(\delta\To {q})"]  \arrow[r, "\tau'"] & (\Pi_\AA\BB)^\II\times\II \arrow[d, "\delta \To {p}_*q"]\\ (\Pi_\AA \BB)_{\epsilon} \arrow[r, "\kappa"'] & \Pi_{\AA^\II\times\II} (\BB_{\epsilon}) \arrow[r, "\tau"'] & (\Pi_\AA \BB)_{\epsilon}
\end{tikzcd}
\]
\end{prop}
\begin{proof}
Lemma \ref{lem:section-square} verifies the commutativity of the left-hand square under weaker hypotheses, while Lemma \ref{lem:section-retraction} proves that the bottom composite $\tau \cdot \kappa$ is the identity. To see that $\tau' \colon \Pi_{\AA^\II \times \II}(\BB^\II \times \II) \To (\Pi_\AA\BB)^\II \times \II$ is a retract of $\kappa' \colon (\Pi_\AA\BB)^\II \times \II \to \Pi_{\AA^\II \times \II}(\BB^\II \times \II)$ we show that 
\eqref{eq:top-retract} is a retract of \eqref{eq:top-section}. This follows easily from Lemma \ref{lem:section-retraction} and the observation that the unlabeled isomorphisms \eqref{eq:canonical-isos} in these diagrams are pairwise inverses.

Finally, to prove the commutativity of the right-hand square in the retract diagram, we must show that the pasted composite
\[
  \begin{tikzcd} \slice{\AA^\II \times \II} \arrow[d, "{({p}^\II \times \II)_*}"'] \arrow[dr, phantom, "\cong"] & \slice{\AA^\II} \arrow[dr, phantom, "\cong"] \arrow[d, "{p}^\II_*" description] \arrow[l, "\varpi^*"']& \arrow[l, "\varpi_*"'] \slice{\AA^\II \times \II} \arrow[d, "{({p}^\II\times\II)_*}" description] \arrow[dr, phantom, "\Downarrow\tau"] & \arrow[l, "\epsilon^*"']\slice{\AA} \arrow[d, "{p}_*"] \\ \slice{\XX^\II \times \II} & \slice{\XX^\II}\arrow[l, "\varpi^*"] & \arrow[l, "\varpi_*"] \slice{\XX^\II \times \II} & \slice{\XX} \arrow[l, "\epsilon^*"] \\
    \slice{\XX^\II \times \II} \arrow[u,equals] \arrow[rr, equals] & \arrow[u, phantom, "\Downarrow\nu"] &\arrow[u, equals] \slice{\XX^\II \times \II} & \slice{\XX}  \arrow[l, "\epsilon^*"] \arrow[u, equals]
  \end{tikzcd}
\]
equals the pasted composite
\[
  \begin{tikzcd}
   \slice{\AA^\II \times \II} & \slice{\AA^\II} \arrow[l, "\varpi^*"'] & \slice{\AA^\II \times \II} \arrow[l, "\varpi_*"'] & \slice{\AA \arrow[l, "\epsilon^*"']} \\
    \slice{\AA^\II \times \II} \arrow[d, "({p}^\II \times\II)_*"'] \arrow[u,equals] \arrow[rr, equals] & \arrow[u, phantom, "\Downarrow\nu"] &\arrow[u, equals] \slice{\AA^\II \times \II} \arrow[d, "({p}^\II \times\II)_*"'] \arrow[dr, phantom, "\Downarrow\tau"] & \slice{\AA}  \arrow[l, "\epsilon^*"] \arrow[u, equals] \arrow[d, "{p}_*"] \\
    \slice{\XX^\II \times \II} \arrow[rr, equals] & & \slice{\XX^\II \times \II} & \slice{\XX} \arrow[l, "\epsilon^*"]
  \end{tikzcd}
  \]
  but, as in the proof of Lemma \ref{lem:section-square}, this follows immediately from Lemma \ref{lem:canonical-counit-iso}.
\end{proof}

\appendix

\section{Evaluation at the generic point}\label{sec:generic}

In this section, we explain the connection between the notion of fibration introduced in Definition \ref{defn:fibration} and evaluation at the generic point in the slice over $\II$. We first explain the meaning of the phrase ``evaluate at the generic point in the slice over $\II$.'' We continue working with the notation for adjoint triples introduced in \eqref{eq:adjoint-triple} and \eqref{eq:adjoint-triple-unit-counit-pairs}

\begin{rmk}\label{rmk:generic-point}
  The diagonal $\delta$ defines a map over $\II$ displayed below-left
    \[
    \begin{tikzcd} \II \arrow[dr, equals] \arrow[rr, "\delta"] & & \II 
    \times \II \arrow[dl, "\pi"] \\ & \II
    \end{tikzcd} \qquad \rightsquigarrow \qquad 
    \begin{tikzcd} \slice{\II} \arrow[r, bend left, equals] \arrow[r, bend 
    right, "{(-)^\II}"'] \arrow[r, phantom, "\Uparrow\delta"] & \slice{\II}
    \end{tikzcd}
    \]
    that  is thought of as the ``generic point'' of $\II$ when considered as 
    an object in $\slice{\II}$: in the slice over $\II$, the domain of 
    $\delta$ is the terminal object, while the codomain of $\delta$ is the 
    object ``$\II$,'' pulled back to this slice. The natural transformation that defines ``evaluation at the generic point in the slice over $\II$''
  is the corresponding restriction map displayed above-right
  from exponentiation with $\II$, considered as an object over $\II$, to exponentiation with the terminal 
  object in the slice over $\II$, the latter being naturally isomorphic to the 
  identity functor. By Remark \ref{rmk:lccc}, this natural transformation is defined by the pasting diagram:
  \[
     \begin{tikzcd}    & & \slice{\II} \arrow[r, equals] & \slice{\II} \\
      & \slice{\II\times\II} \arrow[r, equals] \arrow[dr, phantom, 
  "\Searrow\eta"] & \slice{\II\times\II} \arrow[u, "\pi_*"] \arrow[r, 
  phantom, "\cong"] & ~  \\     \slice{\II} \arrow[r, "\pi^*"] & 
  \slice{\II\times\II} \arrow[u, equals] \arrow[r, "\delta^*"'] & 
  \slice{\II} \arrow[u, "\delta_*"'] \arrow[dr, equals] \\  \slice{\II} 
  \arrow[u, equals] \arrow[rrr, equals] & \arrow[u, 
  phantom, "\cong"] & & \slice{\II} \arrow[uuu, equals]
     \end{tikzcd}
  \]
  Its restriction along $!^* \colon \slice{\ast} \to \slice{\II}$ specializes this ``evaluation at the generic point in the slice over $\II$'' natural transformation to objects that are pulled back to live in the slice over $\II$. We refer to this restricted natural transformation as ``evaluation at the generic point of $\II$.'' Upon pasting with the natural isomorphisms defined by taking iterated mates of the identity transformation associated with the pullback square below-left
  \begin{equation}\label{eq:ev-generic-point-in-slice}
    \begin{tikzcd} \II \times \II \arrow[d, "\pi"'] \arrow[r, "\mu"] \arrow[dr, phantom, "\lrcorner" very near start] & \II \arrow[d, "!"] \\ \II \arrow[r, "\!"'] & \ast
    \end{tikzcd} \qquad \qquad \epsilon_\II^\delta \coloneq 
      \begin{tikzcd} \slice{\ast} \arrow[r, "!^*"] & \slice{\II} \arrow[r, equals]     & \slice{\II} \arrow[r, equals] & \slice{\II} \\ \slice{\II} \arrow[r, "\mu^*"'] \arrow[u, "!_*"] &
        \slice{\II\times\II} \arrow[u, "\pi_*"'] \arrow[ul, phantom, "\cong"] \arrow[r, equals] \arrow[dr, phantom, 
   "\Searrow\eta"] & \slice{\II\times\II} \arrow[u, "\pi_*"] \arrow[r, 
   phantom, "\cong"] & ~  \\ \slice{\II} \arrow[r, "\mu^*"] \arrow[dr, phantom, "\cong"] \arrow[u, equals] &      
   \slice{\II\times\II} \arrow[u, equals] \arrow[r, "\delta^*"'] & 
   \slice{\II} \arrow[u, "\delta_*"'] \arrow[r, equals] & \slice{\II} \arrow[uu, equals]\\ \slice{\ast} \arrow[r, "!^*"'] \arrow[u, "!^*"] & \slice{\II} \arrow[u, "\pi^*"']  \arrow[r, equals]  \arrow[ur, 
   phantom, "\cong"] & \slice{\II} \arrow[u, equals] \arrow[r, equals] & \slice{\II} \arrow[u, equals]  
      \end{tikzcd}
    \end{equation}
  we may regard the ``evaluation at the generic point of $\II$'' natural transformation as a map whose component at $\XX \in \slice{\ast}$ has the form $\epsilon_\II \colon \XX^\II \times \II \to \XX \times \II$. 
    \end{rmk}
  
    The components of the natural transformation \eqref{eq:ev-generic-point-in-slice} lie in the slice over $\II$. The natural transformation \eqref{eq:ev-generic-point-in-slice} may be transposed across the adjunction $!_! \dashv\, !^*$ by post-whiskering with $!_! \colon \slice{\II} \to \slice{\ast}$ and pasting with the counit $\mu$ to define a natural transformation: 
  \begin{equation}\label{eq:generic-point}
     \begin{tikzcd} \slice{\ast} \arrow[r, "!^*"] & \slice{\II} \arrow[r, equals]     & \slice{\II} \arrow[r, equals] & \slice{\II} \arrow[r, "!_!"] & \slice{\ast}\\ \slice{\II} \arrow[r, "\mu^*"'] \arrow[u, "!_*"] &
       \slice{\II\times\II} \arrow[u, "\pi_*"'] \arrow[ul, phantom, "\cong"] \arrow[r, equals] \arrow[dr, phantom, 
  "\Searrow\eta"] & \slice{\II\times\II} \arrow[u, "\pi_*"] \arrow[r, 
  phantom, "\cong"] & ~  \\ \slice{\II} \arrow[r, "\mu^*"] \arrow[dr, phantom, "\cong"] \arrow[u, equals] &      
  \slice{\II\times\II} \arrow[u, equals] \arrow[r, "\delta^*"'] & 
  \slice{\II} \arrow[u, "\delta_*"'] \arrow[r, equals] & \slice{\II} \arrow[uu, equals]\\ \slice{\ast} \arrow[r, "!^*"'] \arrow[u, "!^*"] & \slice{\II} \arrow[u, "\pi^*"']  \arrow[r, equals]  \arrow[ur, 
  phantom, "\cong"] & \slice{\II} \arrow[u, equals] \arrow[r, equals] & \slice{\II} \arrow[u, equals] \arrow[r, "!_!"] \arrow[dr, phantom, "\Searrow\mu"] & \slice{\ast} \arrow[uuu, equals] \arrow[d, equals] \\ \slice{\ast}\arrow[r, equals] \arrow[u, equals] & \slice{\ast}\arrow[u, "!^*"] \arrow[rr, equals] &  &  \slice{\ast}\arrow[u, "!^*"] \arrow[r, equals]  & \slice{\ast}
     \end{tikzcd}
    \end{equation} The effect of the adjoint transposition from \eqref{eq:ev-generic-point-in-slice} to \eqref{eq:generic-point} on the component at $\XX \in \slice{\ast}$ is to post-compose with the projection $\varpi \colon \XX \times \II \to \XX$ and forget that this map lies over $\II$.

  Note that the boundary of the natural transformation \eqref{eq:generic-point} agrees with the boundary of the natural transformation $\epsilon$ defined in \eqref{eq:varpi-epsilon}, and we will show that these natural transformations agree.

    \begin{lem}\label{lem:generic-point} The counit $\epsilon$ of the adjunction $!_!!^* \dashv \, !_*!^*$ is equal to the composite of the ``evaluation at the generic point'' natural transformation followed by the projection away from $\II$. That is, for any $\XX \in \slice{\ast}$, 
      \[\begin{tikzcd} \XX^\II \times \II \arrow[dr, "\epsilon"'] \arrow[r, "\epsilon_\II^\delta"] & \XX \times \II \arrow[d, "\varpi"] \\ & \XX\end{tikzcd}\]
    \end{lem}
    \begin{proof}
      We must show that the pasted composite \eqref{eq:generic-point} agrees with the counit $\epsilon$. Taking mates in the vertical direction of the latter gives the identity, while for the former this yields 
      \[
        \begin{tikzcd} \slice{\ast} \arrow[r, "!^*"] & \slice{\II} \arrow[r, equals]     & \slice{\II} \arrow[r, equals] & \slice{\II} \arrow[r, "!_!"] & \slice{\ast}\\ \slice{\II} \arrow[r, "\mu^*"'] \arrow[from=u, "!^*"'] &
          \slice{\II\times\II} \arrow[from=u, "\pi^*"] \arrow[ul, phantom, "\cong"] \arrow[r, equals] & \slice{\II\times\II} \arrow[from=u, "\pi^*"'] \arrow[r, 
     phantom, "\cong"] & ~  \\ \slice{\II} \arrow[r, "\mu^*"] \arrow[dr, phantom, "\cong"] \arrow[u, equals] &      
     \slice{\II\times\II} \arrow[u, equals] \arrow[r, "\delta^*"'] & 
     \slice{\II} \arrow[from=u, "\delta^*"] \arrow[r, equals] & \slice{\II} \arrow[uu, equals]\\ \slice{\ast} \arrow[dr, phantom, "\Swarrow\mu"] \arrow[r, "!^*"'] \arrow[from=u, "!_!"'] & \slice{\II} \arrow[from=u, "\pi_!"]  \arrow[r, equals]  \arrow[ur, 
     phantom, "\cong"] & \slice{\II} \arrow[u, equals] \arrow[r, equals] & \slice{\II} \arrow[u, equals] \arrow[r, "!_!"]  & \slice{\ast} \arrow[uuu, equals] \arrow[d, equals] \\ \slice{\ast}\arrow[r, equals] \arrow[u, equals] & \slice{\ast}\arrow[from=u, "!_!"'] \arrow[rr, equals] &  &  \slice{\ast}\arrow[from=u, "!_!"'] \arrow[r, equals]  & \slice{\ast}
        \end{tikzcd}\]
Thus, we must show that this pasted natural transformation is the identity. As this natural transformation is the transpose along $!_! \dashv\, !^*$ of the natural transformation 
 \[
  \begin{tikzcd} \slice{\ast} \arrow[r, "!^*"] & \slice{\II} \arrow[r, equals]     & \slice{\II} \arrow[r, equals] & \slice{\II} \\ \slice{\II} \arrow[r, "\mu^*"'] \arrow[from=u, "!^*"'] &
    \slice{\II\times\II} \arrow[from=u, "\pi^*"] \arrow[ul, phantom, "\cong"] \arrow[r, equals] & \slice{\II\times\II} \arrow[from=u, "\pi^*"'] \arrow[r, 
phantom, "\cong"] & ~  \\ \slice{\II} \arrow[r, "\mu^*"] \arrow[dr, phantom, "\cong"] \arrow[u, equals] &      
\slice{\II\times\II} \arrow[u, equals] \arrow[r, "\delta^*"'] & 
\slice{\II} \arrow[from=u, "\delta^*"] \arrow[r, equals] & \slice{\II} \arrow[uu, equals]\\ \slice{\ast}  \arrow[r, "!^*"'] \arrow[from=u, "!_!"'] & \slice{\II} \arrow[from=u, "\pi_!"]  \arrow[r, equals]  \arrow[ur, 
phantom, "\cong"] & \slice{\II} \arrow[u, equals] \arrow[r, equals] & \slice{\II} \arrow[u, equals]  
  \end{tikzcd} \quad = \quad
  \begin{tikzcd} \slice{\ast} \arrow[r, "!^*"] \arrow[d, "!^*"'] \arrow[dr, phantom, "\cong"] & \slice{\II} \arrow[d, "\pi^*"]\arrow[r, equals]  \arrow[dr, phantom, "\cong"]   & \slice{\II} & \slice{\ast} \arrow[d, "!^*"'] \arrow[r, "!^*"] & \slice{\II} \arrow[d, equals] \\ \slice{\II} \arrow[r, "\mu^*"] \arrow[dr, phantom, "\cong"] &      
\slice{\II\times\II}  \arrow[r, "\delta^*"'] & 
\slice{\II} \arrow[from=u, equals]\arrow[r, phantom, "\overset{?}{=}"] & \slice{\II} \arrow[d, "!_!"'] \arrow[r, equals] & \slice{\II} \arrow[d, equals] \arrow[dl, phantom, "\Swarrow\iota"] \\ \slice{\ast}  \arrow[r, "!^*"'] \arrow[from=u, "!_!"'] & \slice{\II} \arrow[from=u, "\pi_!"]  \arrow[r, equals]  \arrow[ur, 
phantom, "\cong"] & \slice{\II} \arrow[u, equals] & \slice{\ast} \arrow[r, "!^*"'] & \slice{\II}
  \end{tikzcd}  \]
  it suffices to show that this equals the whiskered unit $\iota !^*$, i.e., that the pasting equation displayed above right holds. Taking mates in the horizontal direction yields a pasting equation between natural transformations that are each post-whiskered with $!_!$.
\[
  \begin{tikzcd} \slice{\ast} \arrow[from=r, "!_!"'] \arrow[d, "!^*"'] \arrow[dr, phantom, "\cong"] & \slice{\II} \arrow[d, "\pi^*"]\arrow[r, equals]  \arrow[dr, phantom, "\Nwarrow"]   & \slice{\II} & \slice{\ast} \arrow[d, "!^*"'] \arrow[from=r, "!_!"'] & \slice{\II} \arrow[d, equals] \\ \slice{\II} \arrow[from=r, "\mu_!"'] &      
    \slice{\II\times\II}  \arrow[from=r, "\delta_!"] & 
    \slice{\II} \arrow[from=u, equals]\arrow[r, phantom, "\overset{?}{=}"] & \slice{\II} \arrow[d, "!_!"'] \arrow[r, equals] & \slice{\II} \arrow[d, equals] \arrow[ul, phantom, "\Nwarrow\iota"] \\ \slice{\ast}  \arrow[from=r, "!_!"] \arrow[from=u, "!_!"'] & \slice{\II} \arrow[from=u, "\pi_!"]  \arrow[r, equals]  & \slice{\II} \arrow[u, equals] & \slice{\ast} \arrow[from=r, "!_!"] & \slice{\II}
      \end{tikzcd} 
  \]
  This pasting equation holds because the top rows of these natural transformations, where the non-identity cells live, are equal, as can be seen by taking mates once more in the vertical direction.
  \end{proof}

  \begin{prop}
    \label{prop:fibration}
  For any ${p} \colon \AA \to \XX$, the map $\delta\To{p}$ is 
  isomorphic to the Leibniz exponential in the slice over $\II$ of the map 
  ${p} \times \II$ with the diagonal $\delta \colon \II \to \II \times 
  \II$.
  \end{prop}
  \begin{proof}
    By applying the functor $!^* \colon \slice{\ast} \to \slice{\II}$, the 
    map ${p}$ can be pulled back to define a map ${p} \times \II 
    \colon \AA \times \II \to \XX \times \II$ in the slice over $\II$. By Lemma \ref{lem:generic-point},  the pullback \eqref{eq:fibration} factors as below-left
  \[
     \begin{tikzcd} \AA^\II \times \II \arrow[drrr=20, bend left, "\epsilon"] 
  \arrow[drr=5, bend left, "\epsilon_\II^\delta" pos=.7] \arrow[ddr, bend right, 
  "{p}^\II \times \II "'] \arrow[dr, dashed, "\delta\To{p}"]  & & & & & \AA^\II \times \II 
  \arrow[drr=5, bend left, "\epsilon_\II^\delta"] \arrow[ddr, bend right, 
  "{p}^\II \times \II "'] \arrow[dr, dashed, "\delta\To{p}"]  \\ &  
  \AA_{\epsilon} \arrow[d, "\epsilon^*{p}"'] \arrow[r, 
  "({p}\times\II)^*\epsilon_\II^\delta"] \arrow[dr, phantom, 
  "\lrcorner" very near start] & \AA \times \II \arrow[d, 
  "{p}\times\II"'] \arrow[r, "\varpi"] \arrow[dr, phantom, "\lrcorner" 
  very near start]& \AA \arrow[d, "{p}"] &  \rightsquigarrow & &  
  \AA_{\epsilon} \arrow[d, "\epsilon^*{p}"'] \arrow[r, 
  "({p}\times\II)^*\epsilon_\II^\delta"] \arrow[dr, phantom, 
  "\lrcorner" very near start] & \AA \times \II \arrow[d, 
  "{p}\times\II"']  \\ & \XX^\II\times \II  
  \arrow[r, "\epsilon_\II^\delta"] \arrow[rr, bend right=15, "\epsilon"'] & \XX 
  \times \II \arrow[r, "\varpi"] & \XX & & & \XX^\II\times \II  
  \arrow[r, "\epsilon_\II^\delta"'] & \XX 
  \times \II 
     \end{tikzcd}
     \]
     and thus it suffices to consider the pullback diagram above-right which lives in $\slice{\II}$. By Remark \ref{rmk:generic-point}, the top and bottom horizontal components are defined by evaluation with the generic point in the slice over $\II$ for the objects $\AA\times\II$ and $\XX \times \II$ respectively. Thus this diagram defines the Leibniz exponential with the generic point in the slice over $\II$ as claimed.
  \end{proof}
  
\section{Type-theoretic interpretation of the proof}\label{sec:type-theory}

In this section, we re-express the proof given above in type theory. For a clear exposition of the type theory for a locally cartesian closed category see \cite{Newstead}. We first redescribe the retract diagram \eqref{eq:retract}. Suppose a map $p \maps \AA \to \XX$ is classified by a type $\alpha \maps \XX \to \UU$, and a map $q \maps \BB \to \AA$ is classified by a type $\beta \maps \XX. \alpha \to \UU$. The Leibniz exponential $\delta \To p \maps \AA^\II \times \II \to \AA_\epsilon$ then has the following type:
\begin{equation*}
\tminctx{x : \II \to \XX \, , i:\II}{(\delta \To p)}{\left(\dprd{j:\II} \alpha(x(j))\right)\to \alpha(x(i))} 
\end{equation*}
and given a term $\overline{a} : \dprd{j:\II} \alpha(x(j))$ in the same context, we have $(\delta \To p) \, (\overline{a}) \coloneq \overline{a}(i)$. Thus, the map $\delta\To p_*q \colon (\Pi_\AA\BB)^\II\times\II \to (\Pi_\AA\BB)_\epsilon$ has the type
\begin{equation*}
  \tminctx{x : \II \to \XX \, , i:\II}{(\delta \To p_*q)}{\left(\dprd{j : \II}\dprd{\overline{a}: \alpha(x(j))} \beta(x(j),\overline{a}) \right)\to \left(\dprd{a: \alpha(x(i))} \beta(x(i),a)\right)} 
\end{equation*}

The map $\kappa \maps (\Pi_\AA \BB)_{\epsilon} \to  \Pi_{\AA^\II\times\II} (\BB_{\epsilon}) $ 
is given by the following term
\begin{equation*}
\tminctx{x : \II \to \XX \, , i: \II}{\kappa}{\left(\dprd{a: \alpha(x(i))} \beta(x(i),a)\right) \to \left( \dprd{\overline{\overline{a}} \oftype \dprd{k : \II} \alpha(x(k))} \beta(x(i), \overline{\overline{a}}(i))\right)}
\, ,
\end{equation*}
and given a term $w : \dprd{a: \alpha(x(i))} \beta({x(i),a})$ in the same context, we have
\[
\kappa (w) \coloneq \lam{\overline{\overline{a}}} w( \overline{\overline{a}}(i) ) \, .
\]
The map $\kappa' \colon (\Pi_\AA\BB)^\II \times \II \to \Pi_{\AA^\II \times \II} (\BB^\II \times \II)$, constructed from $\kappa$, corresponds to a term 
\begin{equation*}
\tminctx{x : \II \to \XX \, , i: \II}{\kappa'}{\left(\dprd{j : \II}\dprd{\overline{a}: \alpha(x(j))} \beta(x(j),\overline{a}) \right) \to \left(\dprd{\overline{\overline{a}} \oftype \dprd{k : \II} \alpha(x(k))} \dprd{j: \II} \beta(x(j), \overline{\overline{a}}(j))\right)}
\, ,
\end{equation*}
and given a term $v : \dprd{j : \II}\dprd{\overline{a}: \alpha(x(j))} \beta(x(j),\overline{a})$ in the same context, we have 
\[ 
\kappa' (v) \coloneq\lam{\overline{\overline{a}}}\lam{j} v ( j, \overline{\overline{a}}(j)) \, .
\]
The composite pasting of the two isomorphisms in the diagram \eqref{eq:top-section}, evaluated at the component $\epsilon^*(q)$, corresponds to the canonical isomorphism of types 
\begin{equation}
    \label{eq:exchanging_Pis}
\dprd{j: \II}\dprd{\overline{\overline{a}} \oftype \dprd{j : \II} \alpha(x(j))} \beta(x(j), \overline{\overline{a}}(j)) \, \cong \, \dprd{\overline{\overline{a}} \oftype \dprd{j : \II} \alpha(x(j))} \dprd{j: \II} \beta(x(j), \overline{\overline{a}}(j)) 
\end{equation}
induced by changing the order of the $\Pi$-types. Furthermore, we can see that the diagram in Lemma \ref{lem:section-square} commutes since the left-bottom composite is $v \mapsto v(i) \mapsto \lam{\overline{\overline{a}}} v(i, \overline{\overline{a}}(i))$ and the top-right one is 
$v \mapsto \lam{\overline{\overline{a}}}\lam{j} v( j, \overline{\overline{a}}(j)) \mapsto \lam{\overline{\overline{a}}} v(i, \overline{\overline{a}}(i))$  and by function extensionality, we conclude that the two compositions are classified by the same term. 
 
Now we examine the types of $\tau$ and $\tau'$. A section $s$ to the map $\delta \To p$ corresponds to a term 
\begin{equation}\label{eq:section-type}
\tminctx{x : \II \to \XX \, , i: \II}{s}{\alpha(x(i)) \to \dprd{j : \II}\alpha(x(j))}
\end{equation}
such that for every term $a : \alpha(x(i))$ in the same context, $s(a) (i) = a$. The retract $\tau \colon \Pi_{\AA^\II\times\II} (\BB_{\epsilon}) \to (\Pi_\AA \BB)_{\epsilon}$ of $\kappa$ corresponds to the term 
\[ 
\tminctx{x : \II \to \XX \, , i: \II}{\lam{g}\lam{a}g(s(a))}{\left(\dprd{\overline{\overline{a}} \oftype \dprd{k : \II} \alpha(x(k))} \beta(x(i), \overline{\overline{a}}(i))\right) \to \left(\dprd{a: \alpha(x(i))} \beta(x(i),a)\right)
}
\, ,\]
Note that the type of $g(s(a))$ is $\beta(x(i), \overline{\overline{a}}(i))[s(a)/\overline{\overline{a}}] = \beta(x(i),a)$ since $s(a)(i)=a$. To see that this map is indeed a retraction of $\kappa$, observe that for any term $w: \dprd{a: \alpha(x(i))} \beta(x(i),a)$, we have 
\begin{align*}
\lam{g}\lam{a}g(s(a)) \, \circ \, \lam{w}\lam{\overline{\overline{a}}}w(\overline{\overline{a}}(i)) \, (w) & =
(\lam{g}\lam{a}g(s(a))) (\lam{\overline{\overline{a}}} w(\overline{\overline{a}}(i))) \\
&= \lam{a} (\lam{\overline{\overline{a}}}w(\overline{\overline{a}}(i))\, s(a)) \\ 
&= \lam{a} w (s(a)\, (i))  \\
& = \lam{a} w(a) 
\end{align*}
where the first three equalities are the usual reduction by function application and the last equality is by substitution along $s(a) i = a$. Therefore, by function extensionality, we have 
\[
\lam{g}\lam{a}g(s(a)) \, \circ \, \lam{w}\lam{\overline{\overline{a}}}w(\overline{\overline{a}}(i))= \id
\] 

The retract $\tau' \colon \Pi_{\AA^\II \times \II}(\BB^\II \times \II) \to (\Pi_\AA\BB)^\II \times \II$ of $\kappa'$ is given by the following term 
\[
\tminctx{x : \II \to \XX \, , i: \II}{\lam{f}\lam{j}\lam{\overline{a}} f ( s(\overline{a}),  j) }{\left(\dprd{\overline{\overline{a}} \oftype \dprd{k : \II} \alpha(x(k))} \dprd{j:\II} \beta(x(j), \overline{\overline{a}}(j))\right) \to \left(\dprd{j:\II}\dprd{\overline{a}: \alpha(x(j))} \beta(x(j),\overline{a})\right)
}
\]
Here is where we make use of the generic point $i : I$ appearing in the context. We apply the section $s$ when $i$ is replaced by $j$ to a term $\overline{a} : \alpha(x(j))$ to produce a term $s(\overline{a}) : \Pi_{k : \II} \alpha(x(k))$.

That $\tau'$ is a retract of $\kappa'$ follows from the fact that $\tau$ is a retract of $\kappa$ and the invertible 2-cells pasted to the left of $\kappa$ and $\tau$ are pairwise inverses induced by changing the order of $\Pi$-types in \eqref{eq:exchanging_Pis}. More explicitly, for every $v: \dprd{j : \II}\dprd{\overline{a}: \alpha(x(j))} \beta(x(j),\overline{a})$, 
\[ 
\tau' \circ \kappa' \, (v) = 
\tau' (\lam{\overline{\overline{a}}}\lam{j} v(j, \overline{\overline{a}}(j))) = 
\lam{j}\lam{\overline{a}} (\lam{\overline{\overline{a}}}\lam{j} v( j, \overline{\overline{a}}(j))( s(\overline{a}),  j) =
\lam{j} \lam{\overline{a}} v ( j , s(\overline{a})(j)) = 
\lam{j}\lam{v} v (j, \overline{a})
\]
where the last equality holds because $s(\overline{a})(j) = \overline{a}$. Function extensionality implies that $\tau' \circ \kappa' = \id$.  This concludes the description of \eqref{eq:retract}.

Finally, we see that the type theoretic version of our proof matches Coquand's type-theoretic proof, which uses \eqref{eq:retract} to construct a section to the map $\partial\To p_*q$ using a section $t$ to $\partial\To q$.

A section $t$ to the map $\delta \To q$ corresponds to a term 
\[ 
\tminctx{x : \II \to \XX \, ,  i: \II \, ,\overline{\overline{a}} : \dprd{k : \II} \alpha(x(k))}{t}{\beta(x(i), \overline{\overline{a}}(i) ) \to \dprd{j : \II} \beta(x(j), \overline{\overline{a}}(j) )}
\]
such that for every term $b : \beta(x(i),\overline{\overline{a}}(i))$ in the same context, $t(b) (i) = b$. The section $(p^\II \times \II)_*  (t)$ to $(p^\II \times \II)_* (\delta \To q)$ is given by the term 
\[ 
\tminctx{x : \II \to \XX \, , i: \II}{\lam{g}\lam{\overline{\overline{a}}}t(g(\overline{\overline{a}}))}{\left(\dprd{\overline{\overline{a}} \oftype \dprd{k : \II} \alpha(x(k))} \beta(x(i), \overline{\overline{a}}(i))\right) \to \left(\dprd{\overline{\overline{a}} \oftype \dprd{k : \II} \alpha(x(k))} \dprd{j:\II}\beta(x(j),\overline{\overline{a}}(j))\right)
}
\, 
\]

The retract diagram \eqref{eq:retract} constructs a section to $\delta\To p_*q$ given by the composite map $\tau' \circ (p^\II \times \II)_* (t) \circ \kappa$. This map,  when applied to the term $w : \dprd{a : \alpha(x(i))} \beta(x(i),a)$ in the same context, is calculated as follows: 
\begin{align*}
   (\lam{f}\lam{j}\lam{\overline{a}} f(s(\overline{a}),  j) )\, (\lam{g}\lam{\overline{\overline{a}}}t(g(\overline{\overline{a}})))\, (\lam{\overline{\overline{a}}} w( \overline{\overline{a}}(i) )) & =  (\lam{f}\lam{j}\lam{\overline{a}} f(s(\overline{a}),  j) )\,( \lam{\overline{\overline{a}}}t (\lam{\overline{\overline{a}}} w (\overline{\overline{a}}(i)) (\overline{\overline{a}})))  \\
                & =  (\lam{f}\lam{j}\lam{\overline{a}} f ( s(\overline{a}),  j) )\, (\lam{\overline{\overline{a}}} t (w(\overline{\overline{a}}(i))))    \\
                & = \lam{j}\lam{\overline{a}}(\lam{\overline{\overline{a}}} t (w(\overline{\overline{a}}(i))))( s (\overline{a}) ,j))  \\ 
                & = \lam{j}\lam{\overline{a}} t (w(s(\overline{a})(i)))(j)  
\end{align*}
Notice that although $s (\overline{a}) (j) = \overline{a}$, it is not necessary that $s (\overline{a}) \, i = a$ holds. Therefore, the last term above cannot be simplified further. 

\section{Frobenius for structured fibrations}\label{sec:structured}

In this section, we consider the constructive content of our proof that the fibrations of Definition \ref{defn:fibration} are stable under pushforward. We show that under the natural functorial enhancements of the hypotheses on the class of trivial fibrations stated in Theorem \ref{thm:coquand} that we may define a \emph{functorial} Frobenius operator in the sense of Gambino--Sattler \cite{GS}, demonstrating that our proof that the fibrations are closed under pushforward defines a map of structured fibrations. In parallel, we prove, subject to further natural hypotheses on the structured trivial fibrations, that the fibration structures we define on pushforwards of structured fibrations are stable under substitution in a sense we now describe.

``Stability under substitution'' refers to an important property of classes of fibrations: namely, their stability under pullbacks along arbitrary maps. For instance, if our hypothesized class of trivial fibrations is stable under pullback, then the class of fibrations defined by \eqref{eq:fibration} is again stable under pullback since a pullback square as below-left gives rise to a pullback square as below-right:
\begin{equation}\label{eq:basic-stability}
  \begin{tikzcd}
    \CC \arrow[d, "r"'] \arrow[r, "g"] \arrow[dr, phantom, "\lrcorner" very near start] & \AA \arrow[d, "p"] & \arrow[d, phantom, "\rightsquigarrow"] &     \CC^\II \times \II \arrow[r, "g^\II \times \II"] \arrow[d, "\delta\To r"'] \arrow[dr, phantom, "\lrcorner" very near start] & \AA^\II \times \II \arrow[d, "\delta\To p"]  \\ \ZZ \arrow[r, "f"'] & \XX  & ~ & \CC_\epsilon \arrow[r, "{\langle f^\II \times \II, g\rangle}"'] & \AA_\epsilon
  \end{tikzcd}
\end{equation}
Thus if $p$ is a fibration in the sense of Definition \ref{defn:fibration} then so is $r$. 

In constructive settings, it is natural to ask a bit more. There one frequently encounters structured trivial fibrations and structured fibrations encoded by categories over the arrow category. Two further constructive properties --- the pullback stability of the structured fibrations and the existence of canonical fibration structures on pullbacks of structured fibrations --- are encoded in the following axiomatization, introduced by Shulman under the name \emph{notion of fibered structure} \cite[3.1]{Shulman} and by Swan under the name \emph{fibred notion of structure} \cite[3.2]{Swan}.

 For the remainder of this section we work in a fixed locally cartesian closed category $\cE$, if necessary replacing $\cE$ by an equivalent category so that we may choose pullbacks to make $\cod \colon \cE^\2 \to \cE$ into a cloven Grothendieck fibration.

\begin{defn}[{\cite[3.2]{Swan}}]\label{defn:fibred-notion-of-structure}
A \textbf{fibred notion of structure} is a category over the arrow category $u \colon \cF \to \cE^\2$ so that $\cod\cdot u \colon \cF \to \cE$ is a Grothendieck fibration, $u$ is a cartesian functor, and $u$ ``creates cartesian lifts,'' meaning it restricts to define a discrete fibration $u \colon \cF_\cart \to \cE^\2_\cart$ on the subcategories of cartesian arrows:
\[
\begin{tikzcd} \cF \arrow[rr, "u"] \arrow[dr, "\cod\cdot u"'] & & \cE^\2 \arrow[dl, "\cod"] \\ & \cE
\end{tikzcd}
\]
\end{defn}

Swan observes that under these hypotheses, the Grothendieck fibration $\cod \cdot u\colon \cF \to \cE$ can be split so that $u$ strictly preserves the splitting \cite[5.1]{Swan}.

Our first assumption is:
\begin{enumerate}
\item[(STF0)]\label{item:STF0} The trivial fibrations define a fibered notion of structure $u \colon \TFib \to \cE^\2$.
\end{enumerate}

As our aim is to understand the pullback stability of the pushforward of structured fibrations, we use the following definition, which isolates just one of the axioms in Definition \ref{defn:fibred-notion-of-structure}, since it captures the property that is most directly related to stability under substitution.

\begin{defn}\label{defn:stably-structured-fibrations}
A \textbf{category of stably structured fibrations} over $\cE$ is given by a discrete fibration $u \colon \cF_\cart \to \cE^\2_\cart$.
\end{defn}

A fibred notion of structure restricts to define a category of stably structured fibrations, by restricting to the wide subcategory of cartesian arrows. In particular, by (STF0), the trivial fibrations define a category of stable structured fibrations $u \colon \TFib_\cart \to \cE^\2_\cart$. The intention of Definition \ref{defn:stably-structured-fibrations} is to allow us to introduce a category of structured fibrations over the category of arrows and pullback squares without worrying about what more general morphisms of structured fibrations might be. 

The main idea of Definition \ref{defn:stably-structured-fibrations} is that a fibration structure on a map $p$ induces a universal fibration structure on its pullback $r$. If, as suggested by the definition of fibration given in \eqref{eq:fibration}, a fibration structure on $p$ is defined to be a trivial fibration structure on $\delta\To p$, then we have a category of stably structured fibrations as well defined by pullback.

\begin{defn}\label{defn:fib-cat}
Define the category of stably-structured fibrations by
\begin{equation}\label{eq:fib}
\begin{tikzcd} 
\Fib_\cart \arrow[r, "{\delta\To(-)}"] \arrow[d, "u"'] \arrow[dr, phantom, "\lrcorner" very near start] & \TFib_\cart \arrow[d, "u"] \\ \cE^\2_\cart \arrow[r, "{\delta\To(-)}"'] & \cE^\2_\cart
\end{tikzcd}
\end{equation}
\end{defn}

To define a stable functorial Frobenius operator we require additional assumptions on the class of trivial fibrations, each of which is a natural constructive enhancement of the assumptions (TF1)--(TF3) enumerated in the introduction.

\begin{defn}
The \textbf{free retract} is formed by the following pushout of categories:
\[
\begin{tikzcd} \2 \arrow[d, hook, "d^1"'] \arrow[r, "!"] \arrow[dr, phantom, "\ulcorner" very near end] & \1 \arrow[d, "r"] \\ \3 \arrow[r] & \RR
\end{tikzcd}
\]
The category $\RR$ contains two objects---$r$ the ``retract'' and $c$ the ``center''---and two non-identity morphisms--- $d^2 \colon r \to c$ the section and $d^0 \colon c \to r$ the retraction.
\end{defn}

\begin{lem}
The category of maps with a specified section defines a category of stably-structured fibrations:
\[
\begin{tikzcd} \cE^\RR_\cart \arrow[r, hook] \arrow[d, "d^0"'] \arrow[dr, phantom, "\lrcorner" very near start] & \cE^\RR \arrow[d, "d^0"] \\ \cE^\2_\cart \arrow[r, hook] & \cE^\2
\end{tikzcd}\]
\end{lem}
\begin{proof}
A section to a map pulls back to define a section to its pullback \cite[3.4]{Shulman}.
\end{proof}

In order for our stably structured trivial fibrations to have a stable choice of section we require an assumption, giving a constructive enhancement of (TF1):
\begin{enumerate}
\item[(STF1)]\label{item:STF1} There is a map of discrete fibrations:
\[ \begin{tikzcd} \TFib_\cart \arrow[dr, "u"'] \arrow[rr, dashed, "s"] & & \cE^\RR_\cart \arrow[dl, "d^0"] \\ & \cE^\2_\cart
\end{tikzcd}
\]
\end{enumerate}
In other words, we ask that every trivial fibration has a section and that such sections can be chosen so as to be stable under pullbacks.

Assumption (TF2) asks that trivial fibrations, when considered as arrows in a slice category, are stable under pushforward along arbitrary maps. By the Beck-Chevalley property, pushforward can be regarded as a functor between the categories of composable triples and composable pairs of arrows, respectively, and cartesian natural transformations---natural transformations whose components are pullback squares---between them. 

\[
\begin{tikzcd} \cE^\4_\cart \arrow[rr, "\Pi"] \arrow[dr, "\cod"'] & & \cE^\3_\cart \arrow[dl, "\cod"] \\ & \cE
\end{tikzcd}
\]
On objects this functor acts by sending a composable triple $(r,q,p)$ as below-left, to the composable pair $(p_*r, p_*q)$ below-right, where $p_*q$ denotes the action of the pushforward functor on the object $q$ while $p_*r$ denotes the action on the morphism $r$ in the slice over $\dom p$.
\[
\begin{tikzcd} \CC \arrow[rr, "r"] \arrow[dr, "qr"'] & & \BB \arrow[dl, "q"'] & \arrow[d, phantom, "\mapsto"] & \Pi_{\AA} \CC \arrow[rr, "p_*r"] \arrow[dr, "p_*(qr)"']& & \Pi_{\AA}\BB \arrow[dl, "p_*q"] \\ & \AA \arrow[r, "p"] & \XX & ~ & ~& \XX
\end{tikzcd}
\]
This construction is functorial in pullback squares between the composable triples, but is not functorial on the larger category of natural transformations between composable triples. 

Now use a pullback to form the category whose objects are triples comprised of a structured trivial fibration followed by a composable pair of maps out of its codomain, as below-left:
\[
\begin{tikzcd} \TFib_\cart \times_{\cE} \cE_\cart^\3 \arrow[d, "\ev_{01}"'] \arrow[r, "u \times \id"] \arrow[dr, phantom, "\lrcorner" very near start] & \cE_\cart^\4 \arrow[d, "\ev_{01}"]  & & \arrow[d, "\ev_{01}"']\TFib_\cart\times_\cE \cE_\cart^\2 \arrow[r, "u \times \id"] \arrow[dr, phantom, "\lrcorner" very near start] & \cE_\cart^\3 \arrow[d, "\ev_{01}"] \\ \TFib_\cart \arrow[r, "u"'] & \cE_\cart^\2 & & \TFib_\cart \arrow[r, "u"'] & \cE^\2_\cart
\end{tikzcd}
\]
Another pullback defines a similar category whose objects are pairs comprised of a structured trivial fibration together with a map out of its codomain, as above-right. 
Our next assumption is a constructive enhancement of (TF2):
\begin{enumerate}
\item[(STF2)]\label{item:STF2} There is a lift of the pushforward functor:
\[ \begin{tikzcd} \TFib_\cart\times_{\cE} \cE_\cart^\3  \arrow[d, "u \times \id"'] \arrow[r, dashed] & \TFib_\cart \times_{\cE} \cE_\cart^\2 \arrow[d, "u \times \id"] \\ \cE^\4_\cart \arrow[r, "\Pi"'] & \cE^\3_\cart \end{tikzcd}
\]
\end{enumerate}

\begin{ex} 
\label{ex:GS-fibred-notion-of-structure} The condition (TF2) is closely related to the \emph{generalized functorial Frobenius condition} of \cite[\S 6]{GS}. Gambino and Sattler show that if the category of structured trivial fibrations is defined by a right lifting data, then this condition follows from the \emph{functorial Frobenius condition} which dually considers a lift of the pullback to a functor defined on the category of cospans comprised of a structured trivial fibration and a structured left map. In particular, this condition is satisfied when the left maps are the monomorphisms.
\end{ex}

\begin{ex}
In {\cite[\S 2]{A}}, Awodey considers a related setting of a $\tcof \colon \1 \to \Phi$ and shows that for any object $\XX$ there is an associated polynomial monad
  \[ T_\XX \colon \begin{tikzcd} \slice{\XX} \arrow[r, "t_*"] & \slice{\XX \times \Phi} \arrow[r, "\Phi_!"] & \slice{\XX}\end{tikzcd}\]

Algebras for this monad define a fibred notion of structure for which (TF2) holds. In fact, this is a instance of \ref{ex:GS-fibred-notion-of-structure} where $\cJ$ is the category of pullbacks of the universal cofibration $\tcof$ and pullback squares between them.
\end{ex}

Our final assumption enhances (TF3), the stability of trivial fibrations under retracts in the arrow category. By forming a pullback, we may define a category whose objects are retract diagrams of arrows in which the center arrow is a structured trivial fibration:
\[\begin{tikzcd}  \TFib_\cart \times_{\cE^\2_\cart} \cE^{\2 \times \RR}_\cart \arrow[d, "u \times \id"'] \arrow[r] \arrow[dr, phantom, "\lrcorner" very near start] & \TFib_\cart\arrow[d, "u"] \\ \cE^{\RR\times \2}_\cart\arrow[r, "\ev_c"'] & \cE^\2_\cart
\end{tikzcd}
\]
In this category the morphisms are natural transformations whose components define pullback squares between the center and the retract arrows, with the center morphism a morphism of structured trivial fibrations.

Our final assumption is 
\begin{enumerate}
\item[(STF3)]\label{item:STF3} There is a lift of the evaluation at the retract functor:
\[\begin{tikzcd} \TFib_\cart \times_{\cE^\2_\cart} \cE^{\2 \times \RR}_\cart \arrow[d, "u \times \id"'] \arrow[r, dashed, "\ev_r"] & \TFib_\cart \arrow[d, "u"] \\ \cE^{\RR\times \2}_\cart  \arrow[r, "\ev_r"'] & \cE^\2_\cart
\end{tikzcd}
\]
\end{enumerate}

\begin{ex} When the structured trivial fibrations are defined by right lifting data, retracts of such maps inherit canonically specified right lifting data in such a way that (STF3) holds.
\end{ex}

Now we put these assumptions together. Consider the pushforward functor $\Pi \colon \cE^\3_\cart \to \cE^\2_\cart$ that sends a composable pair $(q,p)$ to $p_*q$, pushing forward the arrow $q$ as an object along $p$.

\begin{prop} The pushforward functor lifts to the stably-structured category $\Fib_\cart$:
\[ \begin{tikzcd}  \Fib_\cart  \times_\cE \Fib_\cart \arrow[d, "u \times u"'] \arrow[r, dashed, "\Pi"] & \Fib_\cart \arrow[d, "u"] \\ \cE_\cart^\3 \arrow[r, "\Pi"] & \cE_\cart^\2 \end{tikzcd}
\]
that defines a cartesian functor.
\end{prop}
\begin{proof}
On account of the pullback \eqref{eq:fib}, it suffices to show that the composite functor has a lift:
\[ 
\begin{tikzcd} \Fib_\cart  \times_\cE \Fib_\cart \arrow[d, "u \times u"'] \arrow[rr, dashed, "\delta\To\Pi"] & & \TFib_\cart \arrow[d, "u"] \\ \cE^\3_\cart \arrow[r, "\Pi"'] & \cE^\2_\cart \arrow[r, "\delta\To(-)"'] & \cE^\2_\cart
\end{tikzcd}
\]
Our proof of Theorem \ref{thm:coquand} demonstrates that the bottom-left composite factors as 
\[
\begin{tikzcd} \Fib_\cart \times_\cE \Fib_\cart \arrow[r, "u \times \id"] & \cE^\2_\cart \times_\cE \Fib_\cart \arrow[r] & \cE^{\2 \times \RR}_\cart \arrow[r, "\ev_r"] & \cE^\2
\end{tikzcd}
\]
through the construction of the retract diagram \eqref{eq:retract} associated to a composable pair of morphisms, the second of which is a fibration.

The retract diagram is constructed from the diagram that sends a composable pair $(q,p) \in \cE^\2_\cart \times_\cE \Fib_\cart$ to the diagram 
\[
\begin{tikzcd} \BB^\II \arrow[d, "q^\II"'] & \BB^\II \times \II \arrow[l, "\varpi"'] \arrow[r, "\epsilon"] \arrow[d, "q^\II \times \II"'] & \BB \arrow[d, "q"] \\ 
  \AA^\II \arrow[d, "p^\II"'] & \AA^\II \times \II \arrow[l, "\varpi"'] \arrow[r, "\epsilon"] \arrow[d, "p^\II \times \II"'] & \AA \arrow[d, "p"] \\ \XX^\II & \XX^\II \times \II \arrow[l, "\varpi"] \arrow[r, "\epsilon"'] & \XX
\end{tikzcd}  
\]
The domain and codomains of $\varpi$ and $\epsilon$ are cartesian functors. It follows that their naturality squares for a pair of composable morphisms pull back along any $f \colon \XX' \to \XX$ to define naturality squares for the pulled back morphisms. To build the retract diagram from this data we first form the pullbacks in the right-hand squares to define the maps $\delta\To q$ and $\delta \To p$. By (STF0) and Definition \ref{defn:fib-cat}, this defines a cartesian functor from $\Fib \times_\cE \Fib$ to the category of such diagrams where the pullback corner maps are structured trivial fibrations. 

Then by (STF1) there is a further cartesian functor that extends this diagram by incorporating the canonical section $s$ to $\delta\To p$. A further cartesian functor then constructs the retract diagram \eqref{eq:retract}. Here the Beck-Chevalley condition tells us that constructions of the vertical morphisms and the canonical transformations $\kappa$ and $\kappa'$ are stable under pullback; the Beck-Chevalley condition and (STF1) provide the pullback stability of the construction of the morphisms $\tau$ and $\tau'$. Finally, (STF2) provides a pullback stable trivial fibration structure on the center morphism. This defines the desired lift:
\[
\begin{tikzcd} \Fib_\cart \times_\cE \Fib_\cart \arrow[d, "\id \times u"'] \arrow[r, dashed] & \TFib_\cart \times_{\cE^\2_\cart} \cE^{\2 \times \RR}_\cart \arrow[d] \arrow[r, "\ev_r"] & \TFib_\cart \arrow[d, "u"] \\ \cE_\cart^\2 \times_\cE \Fib_\cart \arrow[r] & \cE_\cart^{\2 \times \RR} \arrow[r, "\ev_r"] & \cE^\2_\cart
\end{tikzcd}
\]
To conclude, (STF3) defines the lifted functor projecting to the retract. 
\end{proof}

\end{document}